\documentclass[final,3p,times]{elsarticle}

\journal{Journal of Computational Physics}

\usepackage[american]{babel}
\usepackage[utf8]{inputenc}
\usepackage{amsfonts, amsmath, amssymb, amsthm, graphicx, float, hyperref, color, mathtools, algpseudocode, algorithm, booktabs, multirow, svg, mathabx, caption}

\graphicspath{{Figures/}}

\usepackage[american]{babel}
\usepackage[utf8]{inputenc}
\usepackage{amsfonts, amsmath, amssymb, amsthm, graphicx, float, hyperref, color, mathtools, algpseudocode, algorithm, booktabs, multirow, svg, mathabx, caption, subcaption}

\usepackage{xcolor}
\usepackage{tcolorbox}

\usepackage[textsize=tiny,textwidth=1.2in]{todonotes}
\usepackage{xspace}

\newcommand{\swlSide}[1]{\xspace}

\newcommand{\revision}[2][]{{#2}} 
\newcommand{\revisionB}[2][]{{#2}}
\newcommand{\revisionC}[2][]{{#2}}

\algblock{Switch}{EndSwitch}
\algblock{Case}{EndCase}
\algnewcommand{\algorithmicSwitch}{\textbf{switch}}
\algnewcommand{\algorithmicEndSwitch}{\textbf{end switch}}
\algnewcommand{\algorithmicCase}{\textbf{Case}}
\algnewcommand{\algorithmicEndCase}{}
\algrenewtext{Case}[1]{\algorithmicCase\ #1:}
\algrenewtext{EndCase}{}

\newtheorem{theorem}{Theorem}

\newtheorem{remark}{Remark}

\newcommand{\Schrodinger}{Schr\"odinger\xspace}

\DeclareMathAlphabet{\mathpzc}{OT1}{pzc}{m}{it}

\DeclareMathOperator{\diag}{diag}

\newcommand{\ctrlA}{{\mathpzc c}}
\newcommand{\jfunc}{{\mathfrak j}}
\newcommand{\NTime}{{N_{\rm T}}}
\newcommand{\NCtrl}{{N_{\rm C}}}
\newcommand{\NPar}{{N_{\rm P}}}
\newcommand{\cost}{{\rm cost}}

\newcommand{\lagP}{\mathfrak{L}}

\newcommand{\Jcal}{\mathcal{J}}
\newcommand{\Lcal}{\mathcal{L}}
\newcommand{\Ocal}{\mathcal{O}}

\newcommand{\ab}{{\mathbf a}}
\newcommand{\bb}{{\mathbf b}}

\newcommand{\eb}{{\mathbf e}}

\newcommand{\yb}{{\mathbf y}}

\newcommand{\ub}{{\mathbf u}}
\newcommand{\vb}{{\mathbf v}}
\newcommand{\wb}{{\mathbf w}}

\newcommand{\re}{{\rm Re}\,}
\newcommand{\im}{{\rm Im}\,}

\newcommand{\lambdab}{\boldsymbol{\lambda}}
\newcommand{\psib}{\boldsymbol{\psi}}

\newcommand{\thetab}{\boldsymbol{\theta}}

\newcommand{\zerob}{\boldsymbol{0}}

\newcommand{\ba}{\begin{array}}
\newcommand{\ea}{\end{array}}
\newcommand{\be}{\begin{equation}}
\newcommand{\ee}{\end{equation}}
\newcommand{\bd}{\begin{displaymath}}
\newcommand{\ed}{\end{displaymath}}

\newcommand{\xb}{\mathbf{x}}

\newcommand{\C}{\mathbb{C}}
\newcommand{\R}{\mathbb{R}}

\newcommand*{\bmat}[1]{\begin{bmatrix} #1 \end{bmatrix}}

\newcommand{\lpar}{\left(}
\newcommand{\rpar}{\right)}

\newcommand{\stepDeriv}[3]{#1_{#2}^{(#3)}}

\newcommand{\thisCiteNeeded}[1]{{\color{red} $\bigoasterisk$}}

\newcommand{\ket}[1]{|#1\rangle}
\newcommand{\bra}[1]{\langle#1|}

\newcommand{\innerprod}[2]{\left\langle #1, #2 \right\rangle}
\newcommand\norm[1]{\left\lVert#1\right\rVert}

\DeclareMathOperator{\tr}{Tr}

\newcommand{\dv}[3][]{\frac{d^{#1}#2}{d{#3}^{#1}}}
\newcommand{\inlinedv}[3][]{{d^{#1}#2}/{d{#3}^{#1}}}
\newcommand{\pdv}[2]{\frac{\partial #1}{\partial #2}}
\newcommand{\inlinepdv}[2]{{\partial #1}/{\partial #2}}

\newcommand{\target}{\text{target}}
\newcommand{\trace}[1]{\tr \left[ #1 \right]}
\newcommand{\Fidelity}{\mathcal{F}}
\newcommand{\Infidelity}{\mathcal{I}}
\newcommand{\GuardPenalty}{\mathcal{G}}

\usepackage{xcolor}
\usepackage{tcolorbox}

\begin{document}

\begin{frontmatter}



\title{High-Order Hermite Optimization: Fast and Exact Gradient Computation in Open-Loop Quantum Optimal Control using a Discrete Adjoint Approach}
\author{Spencer Lee\corref{cor1}\fnref{label1}}
\ead{leespen1@msu.edu}
\cortext[cor1]{Corresponding author}
\affiliation[label1]{organization={Michigan State University, Department of Computational Mathematics, Science, and Engineering \& Department of Physics and Astronomy},
}

\author[]{Daniel Appel\"{o}\fnref{label2}}
\ead{Appelo@vt.edu}
\affiliation[label2]{organization={Department of Mathematics, Virginia Tech},
            addressline={},
            city={Blacksburg},
            postcode={24061},
            state={VA},
            country={USA}}
             

\begin{abstract}
This work introduces the \emph{High-Order Hermite Optimization} (HOHO) method, an open-loop discrete adjoint method for quantum optimal control. Our method is the first of its kind to efficiently compute exact (discrete) gradients when using continuous, parameterized control pulses while solving the forward equations (e.g. \Schrodinger's equation or the Linblad master equation) with an \emph{arbitrarily high-order} Hermite Runge-Kutta method. The HOHO method is implemented in \emph{QuantumGateDesign.jl}, an open-source software package for the Julia programming language, which we use to perform numerical experiments comparing the method to Juqbox.jl. For realistic model problems we observe speedups up to 775x.
\end{abstract}



%
%
%

\end{frontmatter}

{ \hypersetup{hidelinks} }

\section{Introduction}

Quantum optimal control (QOC) is the theory and practice of manipulating quantum mechanical systems. Quantum technologies, including quantum computing, quantum information processing,  quantum simulation, and quantum sensing, depend on the ability to precisely manipulate quantum systems. We focus on the application of QOC to quantum computing algorithms. Such quantum computing algorithms are implemented by applying a series of quantum gates consisting of unitary transformations of a quantum system over a fixed time duration.

A quantum gate is realized in hardware by varying the amplitude of control pulses in time: microwave pulses in the case of superconducting circuits \cite{Krantz_2019,raftery2017directdigitalsynthesismicrowave}, and laser pulses in the case of trapped ions \cite{Sch_fer_2018} or neutral atoms \cite{Levine_2018,Theis_2016}. These control pulses may be generated experimentally using arbitrary waveform generators (AWGs) for microwaves, or acoustic/electro-optic modulators for lasers. In this work we will focus on superconducting circuits, although the methods we describe may be applied to a variety of quantum systems.

For certain quantum gates operating on small, simple systems, analytic control pulses can be found. For larger, more practically useful systems, numerical optimization is used to find control pulses which implement a desired quantum gate. Gate fidelity is measured by an objective function and the control pulses are parameterized by a set of control parameters. 

QOC methods may be open-loop, in which case the quantum system is simulated numerically, or closed-loop, in which case feedback from real experiments is used. For the problem of gate design, closed-loop typically means applying control pulses and measuring the performance of the gate at the end of the pulse application (e.g. by randomized benchmarking \cite{Magesan_2012} or quantum state tomography \cite{QuantumStateTomography}). In general, open-loop methods excel at accurately computing gradients of the objective function with a cost that does not scale with the number of control parameters, whereas closed-loop methods excel at correctly characterizing the behavior of the gate by using experimental results instead of simulated results based on potentially inaccurate theoretical models, and at scaling with the number of qubits. Control pulses found using open-loop control may be further tuned using closed-loop control, so even when closed-loop control is necessary for the final implementation of a gate, open-loop control is still a valuable tool.

The \emph{adjoint method} is a technique for efficiently calculating the gradient of an objective functional that depends on the solution of a differential equation \cite{Givoli_2021_adjoint}. By ``efficiently,'' we mean in particular that the computational cost comes mainly from solving two initial value problems (the forward and adjoint equations), and does not depend significantly on the number of control parameters. Within this framework, there are \emph{continuous adjoint} methods, in which the adjoint equations are derived from the continuous forward equations (optimize then discretize), and \emph{discrete adjoint} methods, in which the adjoint equations are derived from the discretization of the forward equations by a numerical method (discretize then optimize), so that the gradients computed are exact with respect to the numerical method used. This is especially important when optimizing the objective functional using quasi-Newton methods such as L-BFGS \cite{nocedal2006numerical} \revision[because even modest errors in the gradient may ruin the Hessian approximation, so that the advantage of using curvature information is lost]{because moderate errors in the gradient can obscure the approximate Hessians used by quasi-Newton methods \cite{de_Fouquieres_2011}, which results in the advantage of using curvature information being lost}.

The primary result of this work is the introduction of the \emph{High-Order Hermite Optimization} (HOHO) method for QOC, an open-loop discrete adjoint method which is the first of its kind to efficiently compute exact (discrete) gradients when using continuous, parameterized control pulses while solving the forward equations (e.g. \Schrodinger's equation or the Linblad master equation) with an \emph{arbitrarily high-order} numerical method. Namely, we solve the forward equations using Hermite Runge-Kutta methods, and use their relatively simple form when applied to systems of ODEs to derive a procedure that works for every order version of the method. Additionally, we present techniques that we use to make the method more computationally efficient for the specific models and parameterizations found in QOC. The HOHO method is implemented in \emph{QuantumGateDesign.jl} \cite{Lee_QuantumGateDesign_jl_2025}, an open-source software package for the Julia programming language \cite{Julia_2017}.

The most ubiquitous methods for open-loop QOC are
\emph{Gradient Ascent Pulse Engineering} (GRAPE) \cite{KHANEJA2005296, de_Fouquieres_2011} and Krotov's method \cite{Eitan_2011}. These methods rely on the assumption that the control pulses are piecewise-constant, which allows the forward equations for be solved using matrix exponentiation for each time interval over which the control pulses are constant. Consequently, the discretized forward equations are analytically simple, which makes it easy to derive exact expressions for partial derivatives of the objective function. GRAPE is a discrete adjoint method, and updates all the control parameters at once using the gradient, while Krotov's method updates only a single control parameter at a time, with the advantage that the update scheme is guaranteed to converge monotonically.

The piecewise-constant control pulse assumption has several drawbacks. First, the number of control parameters is exactly the number of timesteps taken, so that longer or more ``expressive'' gates result in a more challenging optimization problem in addition to a more challenging numerical simulation problem. Second, because each control pulse amplitude corresponds to a different control parameter, the amplitude may vary significantly between time intervals, resulting in ``high-bandwidth'' control pulses that are difficult to  realize experimentally on an AWG. Third, matrix exponentiation scales poorly with system size, compared to other numerical timestepping methods.

The first and second drawbacks of the piecewise-constant control pulse assumption have essentially been solved by the \emph{Gradient Ascent in Function Space} (GRAFS) method \cite{GRAFS}, which still assumes the control pulses are piecewise-constant, but parameterizes them in terms of a set of basis functions to enable high-resolution control pulses while drastically reducing the number of control parameters. GRAFS uses a modified version of the gradient computation procedure used in GRAPE to keep the gradient computation efficient. To comply with experimental constraints, Slepian sequences \cite{SlepianSequences} are used to synthesize bandwidth-limited controls.

Before GRAFS, The idea of parameterizing the control pulses in terms of basis functions was also explored in the \emph{Chopped Random Basis} (CRAB) method \cite{CRAB_Caneva_2011}, which uses smooth, parameterized control pulses. But as a tradeoff it lacks a way to compute the gradient (except by finite-difference approximation), and instead uses \revision[a direct search]{a derivative-free (zeroth-order) method \cite{DerivativeFreeOptimization}} to optimize the objective function. The \emph{Gradient Optimization of Analytic Controls} (GOAT) method \cite{GOAT_Machnes_2018} also uses smooth, parameterized control pulses, but performs a gradient-based search by differentiating the forward equation with respect to each control parameter (forward-mode differentiation). The GOAT method produces exact gradients, but must solve an initial value problem for every control parameter, so that the gradient computation is not efficient. We also mention the \emph{Pad\'e Integrator COllocation method} (PICO) \cite{trowbridge2023directcollocationquantumoptimal} which optimizes over the quantum states directly (a direct method), implementing \Schrodinger's equation as a constraint in the optimization instead of having it enforced implicitly by a numerical timestepping scheme.

Recently, advances have been made in using discrete adjoint methods with smooth, parameterized control pulses to efficiently compute the gradient when using second-order methods to solve the forward equations. Respectively, Juqbox.jl \cite{petersson2022optimal} does QOC for closed quantum systems using St\"ormer-Verlet as the timestepping method,  and Quandary \cite{gunther2021quandaryopensourcecpackage} does QOC for open quantum systems using the implicit midpoint method as the timestepping method.

In principle, efficient gradient computation can be achieved for any method of solving the forward equations by using reverse-mode automatic differentiation (AD) \cite{Griewank2008_AutomaticDifferentiationPrinciples, Margossian2019_AutomaticDifferentiationReview}. Indeed, AD has been applied in the field of QOC \cite{Nelson2017_QOC_AD_GPU, Abdelhafez2019_QOC_AD_QuantumTrajectories, Schafer2020_QOC_AD}, and works similarly to the discrete adjoint method. The main difference between the two methods is that the discrete adjoint method results in clear, analytic forward and adjoint equations to be solved, which allows for the development of more efficient algorithms and software than is typically implemented by a pure AD approach. 

In particular, reverse-mode AD must store the expression graph and the values
of intermediate variables in that graph. This results in memory storage that
scales \revision{with} the system size and with the number of matrix-vector multiplications (which scales with the number of timesteps taken, and hence also with the stiffness of the system) and makes optimal control of large quantum systems, especially open quantum
systems, infeasible. In contrast, discrete adjoint methods are often able to store intermediate variables in a small number of pre-allocated arrays, drastically reducing the total memory required. This also drastically reduces the number of large (the size of the system) dynamic memory allocations required, which can have a significant impact on performance when solving sparse differential equations (which typically appear in QOC), since the costs of allocating memory and iteratively solving a linear system of equations (e.g. in an implicit method) at each timestep are both $\mathcal{O}(N)$, where $N$ is the size of the \revision[system.]{complex-valued matrix system. For a system of $N_Q$ qubits, we have $N = 2^{N_Q}$.}  \revision[Additionally, the discrete adjoint method differentiates each discretized forward equation to produce a corresponding adjoint equation. This allows each both the forward and adjoint equations to be solved iteratively. In contrast, in most software implementations of AD will differentiate through each step of the iterative solver, resulting in increased memory usage and unnecessary computations compared to the discrete adjoint method.]{\par 

When the solution method for the forward equations involves solving linear systems, for example when using an implicit timestep method, the discrete adjoint method supports using iterative linear solvers such as GMRES \cite{GMRES} more efficiently than AD does. In particular, when solving the forward equations using an iterative linear solver, naive implementations of AD will directly differentiate through the low-level implementation of the solver \cite{margossian2022efficientautomaticdifferentiationimplicit, hovland2024differentiatinglinearsolvers}, which results in increased memory usage and unnecessary computations.
}

Some of the shortcomings of AD are addressed by the semi-automatic
differentiation approach \cite{Goerz2022_SemiAutomaticDifferentiation}. In this
approach, the objective functional is rewritten in terms of intermediate
variables. By choosing the intermediate variables carefully, the computational
graph for can be kept small, reducing the total memory required compared to
automatic differentiation, while still keeping the flexibility in the choice of
the objective functional.

With the HOHO method, we aim to improve the state-of-the-art in discrete adjoint QOC methods using smooth, parameterized control pulses by using arbitrarily high-order numerical methods, enabling the optimal control of stiff quantum systems, which would be extremely computationally expensive to simulate using lower-order methods.

\revision{In principle, St\"ormer-Verlet and the implicit midpoint method can be made higher-order using compositional methods \cite{HairerGeometric, LLNLCompositionalReport}, which divide each timestep into a sequence of forward and backward steps. However, the number of stages in compositional methods grows rapidly, with the tenth-order method requiring 35 stages. Thus if the base timestep method requires solving a linear system of equations (which is true of both the implicit midpoint method and St\"ormer-Verlet when applied to \Schrodinger's equation), then the number of linear solves required to take a full timestep using the compositional method will be much larger than for HOHO. The linear solve is typically the most time consuming part of the solver. Furthermore, since the compositional methods requiring taking steps backward in time, it is unclear if they can be stably applied to problems with damping, which can occur in open quantum systems.}

The rest of the paper is organized as follows. In Section \ref{sec:problem_description}, we provide a mathematical description of quantum computing and quantum optimal control. In Section \ref{sec:numerical_methods}, we describe the Hermite Runge-Kutta methods we use in our numerical simulations, and derive a discrete adjoint method for efficiently computing the gradient of arbitrary objective functionals for arbitrarily high-order Hermite Runge-Kutta methods. This is the main novel result of this work. In Section \ref{sec:correctness_experiments}, we perform numerical experiments on a Rabi oscillator problem to verify the correctness of our method and its implementation. In Section \ref{sec:speedup_experiments}, we use our method to solve a gate design problem for a tripartite system, showcasing the speed of our method. Finally, in Section \ref{sec:conclusions}, we provide concluding remarks and propose ideas for future work.

\section{Problem Description} \label{sec:problem_description}
We now present a more precise description of quantum computing and QOC before detailing the HOHO method.
\subsection{Quantum Computing}
The state of a closed, quantum system is represented by a complex-valued state vector\footnote{In quantum computing literature, states are often represented using bra-ket notation, in which case the state $\psib \in \C^N$ corresponds to $\ket{\psi}$ in a particular basis. Note that $\ket{\psi}$ only represents a physical state: it does not choose a basis to use to represent that state. We switch between the two notations when convenient, but mostly stick to representing states as complex-valued vectors.} $\psib(t) \in \C^N$. Each observable (a quantity which can be physically measured) of the system is associated with a linear operator. The eigenvalues of the operator are the possible outcomes of the measurement, and the corresponding eigenvectors are the states for which that measurement will be observed.

Each entry of the state vector corresponds to an observable state, and the magnitude of the entry squared gives the probability to observe the corresponding state upon measurement. Consequently, the length of the state vector is $\| \psib(t) \|_2 = 1$ at all times because the probabilities of observing each possible outcome must sum to one. This also implies that the time evolution of the state vector is unitary. A state which has a nonzero magnitude for two or more states is said to be in a \emph{superposition} of those states. 

In this work we consider only closed quantum systems. In a quantum computing context, this means considering only time scales short enough that there is no significant interaction between the quantum computer and the environment.

\subsection{Governing Equation: \Schrodinger's Equation}

The time evolution of the state vector from an initial state $\psib_0$ in a closed quantum system is governed by \Schrodinger's equation\footnote{We always choose our units so that $\hbar=1$.}
\begin{equation} \label{eq:schrodinger}
    \frac{d}{dt}\psib(t) = -iH(t)\psib(t),\quad \psib(0) = \psib_0 \in \C^N,
\end{equation}
where $H(t) \in \C^{N \times N}$ is the Hamiltonian of the system. We assume that the Hamiltonian is controlled by a function parametrized by $\NPar$ values collected in the  control vector $\thetab \in \R^{\NPar}$. Accordingly, we write \Schrodinger's equation as
\begin{equation} \label{eq:schrodinger_with_parameters}
    \frac{d}{dt}\psib(t) = -iH(t; \thetab)\psib(t),\quad \psib(0) = \psib_0 \in \C^{N}.
\end{equation}
We make an additional assumption that the Hamiltonian can be decomposed into drift and control Hamiltonians
\begin{equation} \label{eq:drift_and_control_hamiltonians}
    H(t;\thetab) = H_d + H_c(t;\thetab),
\end{equation}
where the drift Hamiltonian governs the behavior of the system if no control is applied.

Quantum computing systems consists of several qubits whose individual states (when not entangled) are represented by a state vector $\psib_{qubit} \in \C^2$. In bra-ket notation, the standard basis elements of $\C^2$  are written as $\eb_1 = \ket{0}$ and $\eb_2 = \ket{1}$. The vector $\ket{0}$ is called the ground state of a qubit, and $\ket{1}$ is called the excited state of a qubit (or first excited state, if the qubit has more than two levels, in which case the qubit is called a \emph{qudit}).

A collection of \revision[$Q$]{$N_Q$} (possibly entangled) qubits is described by a state vector $\psib(t) \in \C^{2^{\revision[Q]{N_Q}}}$. The basis elements may be constructed using Kronecker products of the state vectors of each qubit. E.g. the state corresponding to a 2 qubit system with the first qubit in the ground state and the second qubit in the excited state is $\ket{01} = \eb_1 \otimes \eb_2 \in \C^{2^2}$. The standard ordering of the basis elements in bra-ket notation is $\ket{00}, \ket{01}, \ket{10}, \ket{11}$. A superposition of observable states is represented by a linear combination of the standard basis elements. For example $(\ket{0} + \ket{1})/\sqrt{2}$ corresponds to a qubit in a superposition of the ground and excited states. Upon measurement, there is a 50/50 chance of observing the system to be in either state.

Hamiltonians are often given in terms of raising/lowering operators. For a single $n$-level qudit, the lowering operator takes the form
\[
    a \coloneqq \begin{bmatrix}
    0 & \sqrt{1} &        &            \\
      & \ddots   & \ddots &            \\
      &          & \ddots & \sqrt{n-1} \\
      &          &        & 0          \\
    \end{bmatrix}
    \in \R^{n \times n},
\]
and the raising operator is  given by $a^\dagger$. For a system of $\revision[Q]{N_Q}$ qudits, with the $q$-th qudit having $n_q$ levels, the lowering operator of the subsystem corresponding to the $q$-th qudit is given by
\[
   a_q \coloneqq I_{n_{\revision[Q]{N_Q}}} \otimes \cdots \otimes I_{n_q+1} \otimes a_q \otimes I_{n_q-1} \otimes \cdots \otimes I_{n_1},
   \quad N = \prod_{q=1}^{\revision[Q]{N_Q}} n_q.
\]
As a representative example,  for a system of $\revision[Q]{N_Q}$ transmon qubits in the dispersive limit the drift Hamiltonian is
\begin{equation} \label{eq:dispersive_hamiltonian_model}
    H_d = \sum_{q=1}^{\revision[Q]{N_Q}} \left( 
    \omega_q a_q^\dagger a_q - \frac{\xi_q}{2}a_q^\dagger a_q^\dagger a_q a_q - \sum_{p > q} \xi_{pq} a_p^\dagger a_p a_q^\dagger a_q
    \right).
\end{equation}
Here $\omega_q$ and $\xi_q$ are the ground state transition frequency and self-Kerr coefficient, respectively, of the $q$-th qudit, and $\xi_{pq}$ is the cross-Kerr coefficient between the $p$-th and $q$-th qudits. The self-Kerr coefficients are typically (on IBM devices) an order of magnitude smaller than the ground state transition frequencies, and the cross-Kerr coefficients are typically 2--3 orders of magnitude smaller than the ground state transition frequencies.

\subsection{Optimal Control Problem}
We are interested in optimal control problems where the goal is to control a dynamic system over a time interval $t \in [0,T]$ to minimize some objective function, which may depend on the state of the dynamic system both at the final time $t=T$ and throughout the duration $t \in [0,T]$. The two most common problems in quantum optimal control are \emph{state-to-state transfer} and \emph{gate design} problems. In this paper we focus on gate design problems.

Algorithms for quantum computers are formulated in terms of quantum logic gates, which are unitary operators used to perform logical operations on subsystems of the quantum computer which consist of one or more qubit(s). Examples of single qubit gates include the Pauli-X/Y/Z and Hadamard gates \cite{Nielsen-Chuang}. Two qubit gates include the CNOT (controlled NOT) and SWAP gates. An example of a three qubit gate is the Toffoli gate. A gate is a unitary operator, so in any given basis the gate can be defined as a unitary matrix. Gate design is the problem of finding controls (parameterized by $\thetab$) which result in a Hamiltonian $H(t;\thetab)$ that implements the action of a desired gate in time $T$.

More precisely, we control the time evolution of $E$ state vectors\footnote{For $n$ qubits, $E = 2^n$. If each qubit is modeled as a pure two-level system, then $N=E$. If additional levels are modeled, then $N > E$.} that form a basis of the \revisionC[computational subspace]{possible initial states that the gate may be applied to}:
\begin{equation} \label{eq:schrodinger_gate}
    \frac{d}{dt}U(t) = -iH(t;\thetab)U(t),\quad U(0) = U_0 \in \C^{N \times E}.
\end{equation}
A state-to-state transfer problem is equivalent to a gate design problem where $E=1$.

The action of a gate is defined by its action on the elements of a basis, so to implement a gate operator it is sufficient to find $\thetab$ such that $U(T;\thetab)$ is sufficiently ``close'' to $U_\target \coloneqq G U_0 $, where $G$ is the matrix representation of the gate in the \revisionC[basis of the computational subspace]{computational basis}. We can quantify how close $U(T;\thetab)$ is to $U_\target$ using the \emph{trace fidelity}:\footnote{$\innerprod{A}{B}_F = \trace{A^\dagger B}$ denotes the Frobenius inner product of two matrices. In the more general context where $A$ and $B$ are operators not expressed in a particular basis, this is called the Hilbert-Schmidt inner product. The notation $\innerprod{\boldsymbol{a}}{\boldsymbol{b}}$, without a subscript, denotes the dot product (for real vectors) or Hermitian inner product (for complex vectors).}
\begin{equation} \label{eq:fidelity}
    \Fidelity(U(T;\thetab), U_\target) \coloneqq \frac{1}{E^2}\left| \innerprod{U(T;\thetab)}{ U_\target}_F\right|^2.
\end{equation}
This definition has several motivations. When $U(T;\thetab)$ and $U_\target$ are both unitary, we have $0 \leq \Fidelity \leq 1$. Specifically, $\Fidelity = 1$ when $U(T;\thetab)$ and $U_\target$ are equal, and $\Fidelity = 0$ when the columns of $U(T;\thetab)$ and $U_{\target}$ are orthogonal (i.e. $U(T;\thetab)$ transforms each initial state to a state that is orthogonal to its target state). We also have that $\Fidelity$ is convex with respect to $U(T; \thetab)$ and that the value of $\Fidelity$ is invariant to changes in global phase:
\begin{equation*}
    \Fidelity(U(T;\thetab), U_\target) = \Fidelity(e^{i\revision[\delta]{\phi}_1}U(T;\thetab),e^{i\revision[\delta]{\phi}_2}U_\target),\quad\ \forall \ \revision[\delta]{\phi}_1, \revision[\delta]{\phi}_2 \in \R.
\end{equation*}
This last property is physically important, as the global phase of a state is not physically observable, since we take the magnitude squared of the elements of a state vector when calculating the probabilities of measurement outcomes.

It is common for optimization software to expect an optimization problem to be formulated as a minimization problem, so instead of maximizing the fidelity, we often consider the equivalent problem of minimizing the \emph{trace infidelity}
\begin{equation} \label{eq:infidelity}
    \Infidelity(U(T;\thetab), U_\target) = 1 - \Fidelity(U(T;\thetab), U_\target).
\end{equation}

To illustrate the idea of gate design, we provide the example of implementing a Hadamard gate for a system described by the Hamiltonian
\begin{equation*}
    H(t) = \frac{\omega_0}{2}\sigma_z + c(t)\sigma_x,\quad \textrm{where }\sigma_z \coloneqq \bmat{1 & 0 \\ 0 & -1}, \quad \sigma_x \coloneqq \bmat{0 & 1 \\ 1 & 0}, \quad \omega_0 = 0.1 \textrm{ GHz},
\end{equation*}
by shaping the control pulse $c(t)$ to reach a gate infidelity less than 0.01\%. The resulting pulse and the evolution of the population of each state starting from each initial condition is shown in Figure \ref{fig:hadamard_example}.

\begin{figure}[htb!]
    \centering
    \includegraphics[width=6.25in, height=2.0in]{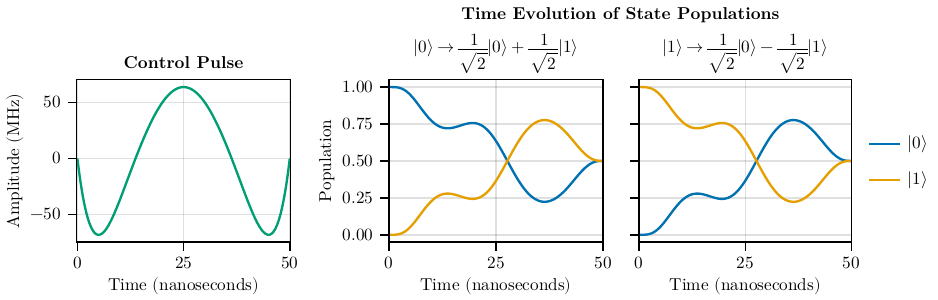} 
    \caption{Control pulse implementing the Hadamard gate for the Hamiltonian $H(t) = 0.5\omega_0 \sigma_z + c(t)\sigma_x$.}
    \label{fig:hadamard_example}
\end{figure}

The infidelity is not necessarily bounded below by 0 when a numerical method is used to approximate $U(T;\thetab)$. This is because the numerical approximation $\tilde{U}(T;\thetab) \approx U(T;\thetab)$ is not necessarily unitary. As an extreme example, if the inner product $\innerprod{U(T;\thetab)}{U_{\target}}_F$ is nonzero, the infidelity can become arbitrarily negative and large if the numerical approximation is $\tilde{U}(T;\thetab) = c U(T;\thetab)$, with $c \gg 1$. Consequently, an optimizer may try to minimize the infidelity \eqref{eq:infidelity} by increasing the spectral norm of $\tilde{U}(T;\thetab)$, which is a completely numerical and non-physical phenomenon. To combat this, we may instead use the \emph{generalized infidelity} \cite{petersson2024timeparallelmultipleshootingmethodlargescale}
\begin{equation} \label{eq:generalized_infidelity}
    \Infidelity_G(U(T;\thetab), U_\target) \coloneqq \frac{1}{E}\norm{U(T;\thetab)}_F^2 - \frac{1}{E^2} \left| \innerprod{U(T;\thetab)}{U_\target}_F \right|^2,
\end{equation} 
which equals the standard infidelity when $U(T;\thetab)$ is unitary, but is \revision[bounded between 0 and 1]{non-negative} even for non-unitary $U(T;\thetab)$. Increasing the spectral norm of $U(T;\thetab)$ increases $\norm{U(T;\thetab)}_F^2$ (because $\norm{A}_2 \leq \norm{A}_F$), which deters an optimizer from exploiting this non-physical phenomenon to minimize the infidelity.  

\revision{We emphasize that our main motivation for using the generalized infidelity is not to force numerical approximations of $U(T;\thetab)$ to be unitary even when the numerical approximation is inaccurate, but rather to avoid misleading objective function values that are near zero (or worse, negative) even though $U(T;\thetab)$ and $U_\textrm{target}$ are not closely aligned. This is especially useful for choosing stopping criterion based on the objective function value, for example choosing to stop the optimization once the (generalized) infidelity falls below a certain threshold, indicating that the gate has reached a desired level of accuracy. In general, the extent to which a numerical approximation of $U(t;\thetab)$ is unitary throughout $t \in [0,T]$ depends on the accuracy of the approximation. Later on in the numerical experiments section of this paper (Section \ref{sec:speedup_experiments}) we study this. The result of that experiment is displayed in Table \ref{tab:cnot3_unitarity}, which shows the convergence of the unitarity of the numerical approximation of $U(t;\thetab)$ as the accuracy of a numerical method increases. }

\begin{table}[htb!]
\centering
\begin{tabular}{rllllllllll}
\toprule 
    & \multicolumn{2}{c}{Order 2} & \multicolumn{2}{c}{Order 4} & \multicolumn{2}{c}{Order 6} & \multicolumn{2}{c}{Order 8} & \multicolumn{2}{c}{Order 10}\\
\cmidrule(lr){2-3} \cmidrule(lr){4-5} \cmidrule(lr){6-7} \cmidrule(lr){8-9} \cmidrule(lr){10-11} Steps & Err & Cvg & Err & Cvg & Err & Cvg & Err & Cvg & Err & Cvg  \\
\midrule 
$2^{10}$ & 2.2(-2) & -4.0 & 6.3(-3) & 1.2 & 3.7(-5) & 7.3 & 4.0(-7) & 7.8 & 2.8(-9) & 9.8 \\
$2^{11}$ & 1.7(-2) & 3.6(-1) & 1.4(-4) & 5.5 & 6.1(-7) & 5.9 & 1.6(-9) & 7.9 & 2.8(-12) & 9.9 \\
$2^{12}$ & 8.3(-3) & 1.1 & 8.6(-6) & 4.1 & 9.7(-9) & 6.0 & 6.6(-12) & 8.0 & 1.9(-13) & 3.9 \\
$2^{13}$ & 1.4(-3) & 2.6 & 5.4(-7) & 4.0 & 1.5(-10) & 6.0 & 1.7(-13) & 5.3 & 1.8(-13) & 3.4(-2) \\
$2^{14}$ & 2.8(-4) & 2.3 & 3.4(-8) & 4.0 & 2.5(-12) & 5.9 & 2.2(-13) & -3.8(-1) & 2.0(-13) & -1.6(-1) \\
$2^{15}$ & 6.7(-5) & 2.1 & 2.1(-9) & 4.0 & 2.3(-12) & 1.1(-1) & 2.2(-12) & -3.3 & 2.1(-12) & -3.4 \\
\bottomrule 
\end{tabular}
    \caption{\revision{Error and convergence rates of the unitarity of the state, as measured by the discrete $L^2$ norm of $\frac{1}{\sqrt{E}}\|U(t;\thetab)\|_F$ across all timesteps, for the three-subsystem example when using the ``best'' control for implementing a CNOT gate in the two qudits plus resonator model (see Figure \ref{fig:cnot3_bestcontrol}), using the Hermite method with orders two through ten.}}
\label{tab:cnot3_unitarity}
\end{table}

We use either the trace infidelity or the generalized infidelity as the main component of our objective function, but we may add additional terms in order to achieve other desirable properties. In our case, we also add a term measuring the population of the ``guard'' states (which are not in the \revisionC[computational]{essential} subspace) throughout the duration of the gate:
\begin{equation} \label{eq:guard_penalty}
    \GuardPenalty(U(\cdot;\thetab)) = \frac{1}{T}\int_0^T \innerprod{U(t;\thetab)}{ WU(t;\thetab)}_F dt.
\end{equation}
Here $W$ is an $N \times N$ matrix whose null space is the \revisionC[computational]{essential} subspace. This is done to motivate the truncation of the Hilbert space, which for most quantum computing architectures has an infinite number of dimensions. By including the guard penalty \eqref{eq:guard_penalty} in our objective function, we prevent ``leakage'' of the state vector into the \revisionC[non-computational subspace]{non-essential (also called the guard subspace), which is not used for computation}. Physically, this means punishing the population of highly excited states. In our models, we typically include one or two guard levels for each qubit. Then, in the computational basis, $W$ is diagonal, with \revisionB[$W_{i,i} = 1$]{$W_{i,i} = 0$} when the $i$-th basis state is an essential state. For a non-essential state $j$, we have $W_{j,j} = w_j \neq 0$, where the weight $w_j$ can be adjusted to suppress the population of certain guard levels more than others.

\subsection{Effect of State Errors on Fidelity} \label{subsec:fidelity_error_analysis}
The error $\mathcal{E}(\delta,\Delta t)$  in the state vector caused by a perturbation $\delta$ in a coefficient of the control and the truncation error $\Delta t^p$ for a $p$-th order accurate numerical method will be $\mathcal{E}(\delta,\Delta t) = \mathcal{O}( C\delta + \tilde{C} \Delta t^p)$. Here $C$ and $\tilde{C}$ are constants that can depend on the solution and the choice of numerical method but not on $\delta$ and $\Delta t$. Here the goal is to determine the control accurately and it is therefore important that \revisionC[$\tilde{C} \Delta t^p \ll C\delta + \tilde{C}$]{$\tilde{C} \Delta t^p \ll C\delta$}, so that the optimization landscape depends more heavily on the real, physical change in the state vector due to the perturbation $\delta$ than on the non-physical change due to the truncation error. Clearly this is easier to achieve for a high-order method. 

\revision[We note, however, that the error in the infidelity is quadratic in the state error. As a result, the state error can generally be allowed to be quadratically larger than the error in the fidelity. For other objectives, the error may be linear in the state error, and hence a higher level of accuracy in the state will be required.]{We note, however, that the error in the infidelity scales quadratically with the state error. Therefore, an error $\mathcal{E}(\delta, \Delta t)$ in the state vector causes an error $\mathcal{E_\mathcal{I}}(\delta, \Delta t)$ in the infidelity which is $\Ocal(\mathcal{E}(\delta, \Delta t)^2)$. As a result, in order to accurately optimize the infidelity up to a tolerance of $\eta$, $\Delta t$ should be chosen so that the truncation error $\tilde{C}\Delta t^p$ is on the order of $\sqrt{\eta}$ or smaller. For other objective functions, the error may scale linearly with the state error. In that case, a higher level of accuracy in the state will be required.}

As an example, we optimize the generalized infidelity of a CNOT gate for a system of two qubits coupled to a resonator bus (which will be described in more detail in Section \ref{sec:speedup_experiments}). The optimization is performed using a fourth-order method to solve \Schrodinger's equation numerically, with the timestep size chosen so that the relative final state error is approximately $10^{-1}$. For each optimization iteration, we solve the same equation again, choosing the timestep size so that the relative final state error is approximately $10^{-7}$\revision{.} In Figure \ref{fig:cnot3_accuracy_verification}, we plot the generalized infidelity as computed by the low-accuracy and high-accuracy solutions, as well as the relative final state error in the low-accuracy solution, as a function of the number of optimization iterations. We see that although after 500 iterations the optimization produces a gate with a generalized infidelity of approximately $10^{-4}$ as computed by the low-accuracy method, the actual generalized infidelity is closer to $10^{-1}$, and the optimization iterations performed after the generalized infidelity had reached $10^{-2}$ were wasted computations which actually increased the generalized infidelity. \revision{Comparing the high-accuracy infidelity to the final state relative error of the low-accuracy method, we see that the infidelity is generally bounded below by the square of the state error, which is consistent with our above prediction that $\Delta t$ should be chosen so that the truncation error is on the order of $\sqrt{\eta}$, where $\eta$ is the desired tolerance in the infidelity.}

\begin{figure}[htb!]
    \centering
    \includegraphics[width=5.75in, height=2.5in]{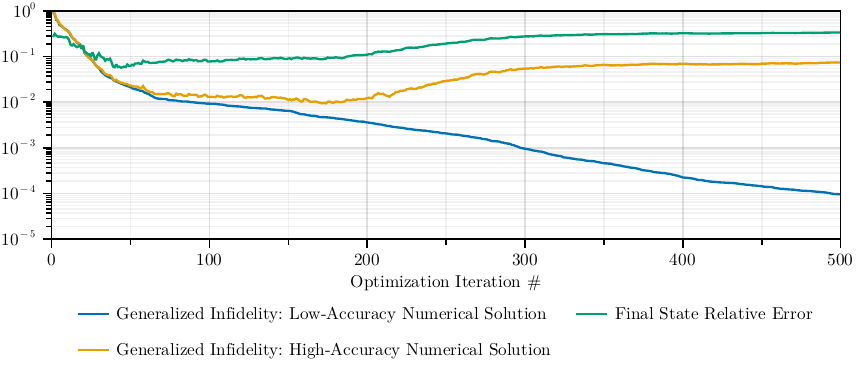}
    \caption{Generalized infidelity of a CNOT gate, as computed by a low-accuracy and high-accuracy numerical method\revision[ as a function of the number of optimization iterations, where the optimization is performed using the low-accuracy method.]{. The generalized infidelity is optimized using a small number of timesteps. For each set of control parameters found at each iteration of the optimization, we also simulate the system using these control parameters and a larger number of timesteps in order to obtain a high-accuracy solution, which is used to compute a more accurate generalized infidelity and the error in the low-accuracy solution at the final time.}}
    \label{fig:cnot3_accuracy_verification}
\end{figure}

\section{Numerical Methods} \label{sec:numerical_methods}
In what follows, we convert \Schrodinger's equation \eqref{eq:schrodinger} into its real-valued formulation, which is commonly done in numerical software. The real-valued form is
\begin{equation*} \label{eq:real_schrodinger_uv}
    \bmat{ \frac{d\ub}{dt} \\ \frac{d\vb}{dt} } = 
    \bmat{ S(t) & K(t) \\ -K(t) & S(t)} \bmat{ \ub \\ \vb },
\end{equation*}
where we have defined $\ub \coloneqq \re(\psib)$, $\vb  \coloneqq \im(\psib)$, $K(t) \coloneqq \re(H(t))$, and $S(t) \coloneqq \im(H(t))$. Because $H(t)$ is Hermitian, we have $K(t)^T = K(t)$ is symmetric, and $S(t)^T = -S(t)$ is anti-symmetric. For convenience, we also define 
\begin{equation*}
    \wb \coloneqq \bmat{\ub \\ \vb}, \quad  A(t) \coloneqq \bmat{ S(t) & K(t) \\ -K(t) & S(t)}
\end{equation*}
so that we can write the real-valued formulation of \Schrodinger's equation more compactly as
\begin{equation} \label{eq:real_schrodinger}
    \frac{d}{dt}\wb(t) = A(t)\wb(t).
\end{equation}
\subsection{Hermite Runge-Kutta Methods}
We now describe Hermite Runge-Kutta methods \cite{Corliss1998HOP_HermiteMethod,gu2020hermite} and how they can be used to efficiently compute discrete adjoints. We follow the description in  \cite{Corliss1998HOP_HermiteMethod} which explicitly considered the case of a variable coefficient linear system of ordinary differential equations and thus can be directly applied to the problem at hand. 

As discussed above, the models we are interested in take the form 
\revision[
\begin{equation} \label{eq:adjoint_state_diffeq}
    \frac{d}{dt}\wb(t) = A(t;\thetab)\wb(t) \quad 0 \leq t \leq T,\quad \wb(0) = \wb_0 \in \R^N.
\end{equation}
]{
\begin{equation} \label{eq:adjoint_state_diffeq}
    \frac{d}{dt}\wb(t) = A(t;\thetab)\wb(t) \quad 0 \leq t \leq T,\quad \wb(0) = \wb_0 \in \R^{2N}.
\end{equation}
}
Here the matrix $A(t;\thetab)$ depends on time and a vector $\thetab$
which parameterizes the time-dependence of the controls. For each simulation of
(\ref{eq:adjoint_state_diffeq}) the parameter vector $\thetab$ will be fixed
and for notational convenience we suppress it in this sub-section and write
$A(t;\thetab) = A(t)$. 

To approximate solutions to (\ref{eq:adjoint_state_diffeq}) we introduce the
equidistant grid $t_n = n \Delta t, \, n = 0,\ldots,\NTime, \, \Delta t = T
/ \NTime$. We denote the numerical approximation to the solution
$\wb(t_n)$ by $\wb_n$. 

Now, to compute $\wb_{n+1}$ from initial data $\wb_n$ we first recast \eqref{eq:adjoint_state_diffeq} as
\begin{equation*}
\wb(t_{n+1}) = \wb(t_n) + \int_{t_n}^{t_{n+1}} \frac{d \wb (t)}{dt} dt,
\end{equation*}
replace $\wb(t_n)$ with $\wb_n$, and then approximate the integral by quadrature.

One way of constructing a Runge-Kutta method is to introduce $s$ nodes between $t_{n}$ and $t_{n+1}$ and then integrate the interpolant of $\frac{d \wb (t)}{dt}$ exactly. This is also the approach in a Hermite Runge-Kutta method where the  $s=p+q$ nodes are split into two sets of nodes containing  $p$ and $q$ nodes respectively.  Unlike most Runge-Kutta methods, here there are no nodes in the interior of the interval $[t_n, t_{n+1}]$ and instead $p$ nodes coalesce at $t_n$  and $q$ nodes  coalesce at $t_{n+1}$. Then the interpolation polynomial becomes a Hermite interpolation polynomial of degree $p+q-1$. That is 
\[
\frac{d \wb (t)}{dt} \approx \sum_{j = 0}^{p-1} \lagP_{0,j}(t) \frac{d^{j+1} \wb (t_n)}{dt^{j+1}} + \sum_{j = 0}^{q-1} \lagP_{1,j}(t) \frac{d^{j+1} \wb (t_{n+1})}{dt^{j+1}},
\]
where $\lagP_{0,j}(t)$ and $\lagP_{1,j}(t)$ are the unique degree $p+q-1$ polynomials (sometimes called generalized Lagrange polynomials, see \cite{DQB}) defined through the conditions 
\begin{multline*}
\frac{d^k \lagP_{0,j}(t_n)}{dt^{k}} = \delta_{kj},\quad \frac{d^l \lagP_{1,j}(t_{n+1})}{dt^{l}} = \delta_{lj},\quad \frac{d^l \lagP_{0,j}(t_{n+1})}{dt^{l}} = 0,\quad \frac{d^k \lagP_{1,j}(t_{n})}{dt^{k}} = 0, \\
j = 0, \dots ,p+q-1,\quad  k = 0,\dots, p-1,\quad l = 0,\dots,q-1.  
\end{multline*}

\revision{Here, $\delta_{ij}$ denotes the Kronecker delta, which equals $1$ if $i = j$ and $0$ otherwise.} After evaluating the integrals of $\lagP_{0,j}(t)$ and $\lagP_{1,j}(t)$ the resulting
one-step method \cite{Corliss1998HOP_HermiteMethod} takes the form
\begin{equation} \label{eq:the_method}
    \sum_{j=0}^{q}(-1)^j c_j^{qp} \frac{\stepDeriv{\wb}{n+1}{j}}{j!}\Delta t^j = 
    \sum_{j=0}^{p} c_j^{pq} \frac{\stepDeriv{\wb}{n}{j}}{j!}\Delta t^j,
\end{equation}
where we have introduced the shorthand notation 
$\stepDeriv{\wb}{n}{j}  =  \frac{d^{j} \wb (t_n)}{dt^{j}}$. Here the coefficients $c_j^{pq}$ can be computed explicitly:
\begin{equation*}
    c_{j}^{pq} = \binom{p}{j} \bigg/ \binom{p+q}{j} = \frac{p!(p+q-j)!}{(p+q)!(p-j)!}.
\end{equation*}

For the case of a closed quantum systems we restrict ourselves to methods with $p=q$, for which this method is A-stable. But we note that for open quantum systems with subsystems that decohere fast, L-stable methods with $q > p$ could be explored. 

\subsection{Computing Higher Order Derivatives and Solving The Linear System of Equations} \label{sec:computing_higher_order_derivatives}
For a general non-linear ODE, computing the derivatives
$\stepDeriv{\wb}{n}{j+1}$ can be cumbersome. But when the ODE is governed by
\eqref{eq:adjoint_state_diffeq}, this computation can be done efficiently and
recursively \cite{Corliss1998HOP_HermiteMethod,icosahom2014} by differentiating
\eqref{eq:adjoint_state_diffeq} $j$ times and using the generalized 
Leibniz rule: 
\begin{equation}\label{eq:recursion}
    \stepDeriv{\wb}{n}{j+1} = \sum_{i=0}^{j}\binom{j}{i} \stepDeriv{A}{n}{j-i} \stepDeriv{\wb}{n}{i},\quad j=0,\dots,q-1.
\end{equation}
If we define the linear operators 
\begin{equation}\label{eq:linear_derivative_matrix}
    D_{n,0} = I,\quad D_{n,j+1} = \sum_{i=0}^j  \binom{j}{i} \stepDeriv{A}{n}{j-i} D_{n,i},\quad j=0,\dots,q-1,
\end{equation}
we observe that $\stepDeriv{\wb}{n}{j} = D_{n,j}\wb_n$, and hence the method
\eqref{eq:the_method} can be rewritten as
\begin{equation} \label{eq:the_method_multi_mat}
    \left( \sum_{j=0}^{q}(-1)^j c_j^{qp} \frac{1}{j!} \Delta t^j D_{n+1,j} \right) \wb_{n+1} = 
    \left( \sum_{j=0}^{p} c_j^{pq} \frac{1}{j!} \Delta t^j D_{n,j}  \right) \wb_{n},
\end{equation}
so that a single timestep consists of solving a single linear system of \revision[$N$]{$2N$} equations:
\begin{equation}\label{eq:timestep_matrices}
    \wb_{n+1} - L_{n+1,q} \wb_{n+1} = \wb_n + R_{n,p} \wb_n,
\end{equation}
where 
\begin{equation}\label{eq:the_method_single_mat}
    L_{n+1,q} \coloneqq - \sum_{j=1}^{q}(-1)^j c_j^{qp} \frac{1}{j!} \Delta t^j D_{n+1,j},\quad
    R_{n,p} \coloneqq \sum_{j=1}^{p} c_j^{pq} \frac{1}{j!} \Delta t^j D_{n,j}.
\end{equation}

We can either form the \revision[$N \times N$]{$2N \times 2N$} matrices $L_{n+1,q}$ and $R_{n,p}$ explicitly by explicitly forming the matrices $D_{n,j}$ and then adding them as in \eqref{eq:the_method_single_mat}, or we can compute $L_{n+1,q}\wb_{n+1}$ and $R_{n,p}\wb_n$ in a ``matrix-free'' fashion by applying $D_{n+1,j}$ and $D_{n,j}$ to $\wb_{n+1}$ and $\wb_n$, and adding the resulting vectors. Furthermore, by storing $D_{n+1,j}\wb_{n+1}$ and $D_{n,j}\wb_n$ after we compute them in the first $j$ terms of the summations in \eqref{eq:the_method_multi_mat}, we can compute $D_{n+1,j+1}\wb_{n+1}$ and $D_{n,j+1}\wb_n$ with only an additional $j$ matrix-vector multiplications. Therefore it costs $q(q+1)/2$ and $p(p+1)/2$ matrix-vector multiplications to compute  $L_{n+1,q}\wb_{n+1}$ and $R_{n+1,q}\wb_n$, respectively, so the cost scales \emph{quadratically} with the order of the method.

In this paper, we always use the matrix-free version. We store only the matrices $\stepDeriv{A}{n}{j}$, which are sparse for most quantum systems, and solve \eqref{eq:timestep_matrices} iteratively, using Algorithm \ref{alg:forward_timestep} to apply $-L_{n+1,q}$ and $R_{n,p}$.

\begin{algorithm}[htb!]
    \caption{Applying $-L_{n,q}$ (or $R_{n,p})$ to $\wb_n$.}
    \label{alg:forward_timestep}
\begin{algorithmic}
    \State For the system of ODEs $\dv{\wb}{t} = A(t;\thetab)\wb$, compute $-L_{n,q}\wb_n$ (or $R_{n,p}\wb_n$), as required for the timestepping rule \eqref{eq:timestep_matrices}. When $p=q$, the timestepping method is A-stable and has order of accuracy $2p$.
    \State
        \State $\frac{\wb_n^{j}}{j!} \gets \frac{1}{j}\sum_{i=0}^{j-1} \frac{1}{(j-1-i)!}A_n^{(j-1-i)}\frac{\wb_n^{(i)}}{i!}$ for $j=1,\dots,q$.
    \State \textbf{return } $\sum_{j=1}^{q} (-1)^j c_j^{qp} \Delta t^j \frac{\wb_n^{(j)}}{j!}$ \Comment{To apply $R_{n,p}$ instead of $L_{n,q}$, swap $p$ and $q$ and remove the $(-1)^j$ factor.}
\end{algorithmic}
\end{algorithm}

\begin{remark}
Although we do not construct the linear operators $L$ and $R$, the simple form of the timestepping \eqref{eq:timestep_matrices} when
written in terms of these linear operators will make deriving the discrete adjoint in Section \ref{sec:discrete_adjoint} simple.
\end{remark}

\subsection{Control Pulse Ansatz}
It should be noted for the quantum optimal control problem that in addition to the matrix  $A$ being the sum of a drift and control matrix, as in \eqref{eq:drift_and_control_hamiltonians}, the time-dependence takes the specific form
\begin{equation} \label{eq:control_pulse_ansatz}
    A(t) = A_{d} + A_{\ctrlA}(t),\ \ \text{where}\quad  A_\ctrlA(t) = \sum_{j=1}^{N_C} \ctrlA_j(t) A_{\ctrlA,j},\
    \quad
    \ctrlA_j(t) = \sum_{k=1}^{n_j} \theta_{j,k} \ctrlA_{j,k}(t), \quad \sum_{j=1}^{\NCtrl} n_j = \NPar.
\end{equation}
Here the multi-index $\theta_{j,k}$ simply indicates the $k$-th control parameter for the $j$-th control pulse (it is also assumed here that each control parameter affects only one control pulse, though in principle this need not be the case). Each $\ctrlA_j(t)$ is a scalar function expressed as a linear combination of pre-defined control functions which we call the \emph{control pulse ansatz}. A nonlinear dependence could also be considered, but constructing the control pulses this way makes it easy to compute their time derivatives and their gradients with respect to the control parameters. It can be done in the following way:
\begin{equation*}
\dv[i]{\ctrlA_{j}(t)}{t}  = \sum_{k=1}^{n_j} \theta_{j,k}\dv[i]{\ctrlA_{j,k}(t)}{t}.
\quad
\pdv{}{\theta_{j,k}}\dv[i]{\ctrlA_j(t)}{t}  = \dv[i]{\ctrlA_{j,k}(t)}{t},
\end{equation*}
In addition, we can choose basis functions $\ctrlA_{j,k}$ whose time derivatives can be easily computed.

The only requirements with respect to the Hermite timestepping method are that we may compute the value of the control pulses and their time derivatives, as well as their gradients with respect to the control parameters, and that the control pulses are sufficiently smooth to match the order of accuracy of the method being used. In principle, any control pulse ansatz may be used.

As our particular choice of control pulse ansatz, we follow the approach of \cite{petersson2022optimal} and use B-splines (basis splines) \cite{deboorSplines2009} multiplied by carrier waves, noting the importance of good frequency selection in achieving low infidelities with a small number of control parameters \cite{CRAB_Caneva_2011}. B-splines have several properties that make them suitable for use as a control pulse ansatz. 1) The partition of unity property ensures that the maximum value of a B-spline curve does not exceed the maximum control coefficient of the basis splines. Therefore, a maximum amplitude of the control pulses can be limited by placing a maximum absolute value on the control coefficients as a constraint in the optimization problem. 2) The local support property makes the pulse parametrization intuitive, since each control coefficient only affects the B-spline curve within a finite, compact region. This also makes it easy to design control pulses which begin and end with amplitude zero, which can be experimentally desirable. 3) The number of continuous derivatives of the B-spline curve increases with the order of basis splines used. Consequently, we can build arbitrarily smooth control pulses by using higher-order basis splines. 4) Finally, there are established, robust, and efficient software implementations of the evaluation of basis splines and B-spline curves \cite{deboorSplines2009}.

\subsection{Discrete Adjoint} \label{sec:discrete_adjoint}
A discrete adjoint method for evaluating derivatives of functions of discrete sequences implicitly defined through forward difference equations has been given by Betancourt et al. \cite{betancourt2020discrete}. We extend that treatment here to be compatible with our method. Let the objective function be defined by the discrete functional
\begin{equation*}
    \Jcal(\thetab) = \sum_{n=0}^{\NTime} \jfunc_n(\wb_n, \thetab).
\end{equation*}
Here we describe the case where the state is a vector (e.g. a state preparation problem), but the result can easily be extended to the matrix (gate design) case. 

The derivatives of $\Jcal$, which we want to compute in order to perform a gradient-based search, are given by
\begin{equation}\label{eq:objective_gradient}
    \dv{\Jcal}{\theta_k} = \sum_{n=0}^N \pdv{\jfunc_n}{\theta_k} 
    + \lpar\pdv{\jfunc_n}{\wb_n}\rpar^T \dv{\wb_n}{\theta_k}.
\end{equation}
The $\inlinepdv{\jfunc_n}{\theta_k}$ terms are zero unless the objective function depends explicitly on $\thetab$, e.g. if it includes a regularization term $\| \thetab \|_2$. For simplicity, we will assume that each $\jfunc_n$ only depends $\thetab$ implicitly through the dependence of $\wb_n$ on $\thetab$. The trajectory of $\wb$ is given by our timestepping rule:
\begin{equation} \label{eq:constraint}
    \wb_{n+1} - L_{n+1,q}\wb_{n+1} = \wb_n + R_{n,p}\wb_n, \quad n=0,1,\dots,\NTime-1.
\end{equation}

Computing $d\Jcal/d\thetab$ directly using \eqref{eq:objective_gradient}
requires computing the sensitivities $d\wb_n/d\theta_k$ for each control
parameter, which requires us to evolve a \Schrodinger equation for each
control parameter. Instead, we treat the method as a constraint in the method
of Lagrange multipliers, giving us the Lagrangian
\begin{align} \label{eq:DiscLagrangian}
    \Lcal(\thetab) &= \Jcal(\thetab) + 
    \sum_{n=0}^{\NTime-1} \innerprod{\wb_{n+1} - L_{n+1,q}\wb_{n+1} - \wb_n - R_{n,p}\wb_n}{\lambdab_{n+1}} \\
    &= \jfunc_\NTime + \sum_{n=0}^{\NTime-1} \jfunc_n + \innerprod{\wb_{n+1}  - L_{n+1,q}\wb_{n+1} - \wb_n - R_{n,p}\wb_n}{\lambdab_{n+1}}.
\end{align}

If the constraints \eqref{eq:constraint} are satisfied, then $\Lcal(\thetab) =
\Jcal(\thetab)$ for any value of the Lagrange multipliers $\lambdab_n$. Because the constraints are
defined using the exact timestepping scheme we use to compute each $\wb_n$, we
may assume the constraints are always satisfied and therefore $\Lcal(\thetab)
\equiv \Jcal(\thetab)$. Taking $\dv{}{\theta_k}$ and making use of the fact
that each $\lambdab_n$ has no explicit dependence on $\thetab$ (we are free to
choose $\lambdab_n$ as we like), we get
\begin{equation}\label{eq:gradient_of_discL}
\dv{\Lcal}{\theta_k} = \dv{\Jcal}{\theta_k} = 
    \dv{\jfunc_\NTime}{\theta_k} + \sum_{n=0}^{\NTime-1} \dv{\jfunc_n}{\theta_k} + \innerprod{\dv{\wb_{n+1}}{\theta_k} - \dv{}{\theta_k}\left(L_{n+1,q}\wb_{n+1}\right) - \dv{\wb_n}{\theta_k} - \dv{}{\theta_k}\left(R_{n,p}\wb_{n}\right)}{\lambdab_{n+1}}.
\end{equation}
Our goal is to rewrite \eqref{eq:gradient_of_discL} and choose $\lambda_n$ such that they cancel out the
sensitivities $d\wb_n/d\theta_k$, leaving us with a gradient  is much
easier to compute.

\begin{theorem} \label{thm:discrete_adjoint}
Let the numerical approximation to \Schrodinger's equation
\eqref{eq:schrodinger} satisfy the timestepping rule
\eqref{eq:constraint}. Then the exact gradient of the discrete cost function 
\eqref{eq:DiscLagrangian} with respect to the numerical method
\eqref{eq:the_method} can be computed by
\begin{equation} \label{eq:gradient_accumulation}
\dv{\Jcal}{\theta_k} = 
-\sum_{n=0}^{\NTime-1} 
\innerprod{
    \pdv{L_{n+1,q}}{\theta_k}\wb_{n+1} + \pdv{R_{n,p}}{\theta_k}\wb_n
}
{
    \ \lambdab_{n+1}
},
\end{equation}
where $\lambdab_n$ is chosen to satisfy the terminal condition and adjoint equations\footnote{We follow the convention that the partial derivative of a scalar with respect to a column vector is a row vector.}
\begin{align}
    \lambdab_\NTime - (L_{\NTime,q})^T \lambdab_\NTime
&= -\lpar \pdv{\jfunc_\NTime}{\wb_\NTime} \rpar^T \label{eq:terminal_condition}, \\
    \lambdab_n - (L_{n,q})^T \lambdab_n
    &= \lambdab_{n+1} + (R_{n,p})^T \lambdab_{n+1}
    - \lpar \pdv{\jfunc_n}{\wb_n}  \rpar^T,\quad (n = \NTime-1,\dots,1). \label{eq:adjoint_equations}
\end{align}
\end{theorem}
To this end we  refer to the solving of the linear systems in \eqref{eq:terminal_condition} and \eqref{eq:adjoint_equations} as the \textbf{adjoint evolution} and the explicit computation of the gradient in equation \eqref{eq:gradient_accumulation} as the \textbf{gradient accumulation}. The entire procedure of performing the forward evolution, adjoint evolution, and gradient accumulation is summarized in Algorithm \ref{alg:discrete_adjoint}.
\begin{proof}
Applying the chain rule to \eqref{eq:gradient_of_discL} yields

\begin{multline*}
\dv{\Jcal}{\theta_k} = 
 \pdv{\jfunc_\NTime}{\wb_\NTime}\dv{\wb_\NTime}{\theta_k}
+ \sum_{n=0}^{\NTime-1} 
\pdv{\jfunc_n}{\wb_n}\dv{\wb_n}{\theta_k} \\
+\sum_{n=0}^{\NTime-1} \innerprod
{
    \dv{\wb_{n+1}}{\theta_k} - 
    \pdv{\left(L_{n+1,q}\wb_{n+1}\right)}{\theta_k}  -
    \pdv{\left(L_{n+1,q}\wb_{n+1}\right)}{\wb_{n+1}} \dv{\wb_{n+1}}{\theta_k} -
    \dv{\wb_n}{\theta_k} - 
    \pdv{\left( R_{n,p}\wb_n \right)}{\theta_k}  -
     \pdv{\left( R_{n,p}\wb_n \right)}{\wb_{n}}\dv{\wb_n}{\theta_k}
}
{\lambdab_{n+1}}
.
\end{multline*}
For a matrix vector product $M\xb$, we have $\partial(M\xb)/\partial\xb = M$.
Also, the state vectors have no explicit dependence on $\thetab$ (the
dependence is only implicit), so the above expression for $\inlinedv{\Jcal}{\theta_k}$ simplifies to
\begin{equation*} 
 \pdv{\jfunc_\NTime}{\wb_\NTime}\dv{\wb_\NTime}{\theta_k} \\
+ \sum_{n=0}^{\NTime-1} 
\pdv{\jfunc_n}{\wb_n}\dv{\wb_n}{\theta_k}
+ \innerprod
{
    \dv{\wb_{n+1}}{\theta_k} - 
    \pdv{L_{n+1,q}}{\theta_k}\wb_{n+1}  -
    L_{n+1,q} \dv{\wb_{n+1}}{\theta_k} -
    \dv{\wb_n}{\theta_k} - 
    \pdv{R_{n,p}}{\theta_k}\wb_n   -
    R_{n,p}\dv{\wb_n}{\theta_k}
}
{\lambdab_{n+1}}
.
\end{equation*}
We rearrange the terms to make it clear how to choose $\lambdab$ to cancel out the sensitivities:
\begin{multline*}
\dv{\Jcal}{\theta_k} = 
     \pdv{\jfunc_\NTime}{\wb_\NTime}\dv{\wb_\NTime}{\theta_k} +
    \innerprod{\dv{\wb_\NTime}{\theta_k} - L_{\NTime,q} \dv{\wb_{\NTime}}{\theta_k}}{\lambdab_\NTime} \\
     + \sum_{n=0}^{\NTime-1}
\pdv{\jfunc_n}{\wb_n} \dv{\wb_k}{\theta_k} + 
    \innerprod{ \dv{\wb_n}{\theta_k} - L_{n,q}\dv{\wb_n}{\theta_k}}{ \lambdab_n } -
    \innerprod{ \dv{\wb_n}{\theta_k} + R_{n,p}\dv{\wb_n}{\theta_k} }{ \lambdab_{n+1}  } - \innerprod{\pdv{L_{n+1,q}}{\theta_k}\wb_{n+1} + \pdv{R_{n,p}}{\theta_k}\wb_n}{\lambdab_{n+1}}.
\end{multline*}
Now, making use of the identities $\innerprod{\boldsymbol{a}
\boldsymbol{b}}{\boldsymbol{c}} = \innerprod{
    \boldsymbol{b}}{\boldsymbol{a}^T \boldsymbol{c}}$ and $\boldsymbol{a}^T
\boldsymbol{b} = \boldsymbol{b}^T \boldsymbol{a}$, we get
\begin{multline}
\dv{\Jcal}{\theta_k} = 
 \innerprod{\dv{\wb_\NTime}{\theta_k}}{ 
    \lpar \pdv{\jfunc_\NTime}{\wb_\NTime} \rpar^T + \lambdab_\NTime - \lpar L_{\NTime,q} \rpar^T \lambdab_\NTime
} \label{eq:discrete_adjoint_before_cancellation} \\
 + \sum_{n=0}^{\NTime-1}
\innerprod{\dv{\wb_k}{\theta_k}}{
    \lpar \pdv{\jfunc_n}{\wb_n}  \rpar^T
    +  \lambdab_n - \lpar L_{n,q} \rpar^T \lambdab_n
    - \lambdab_{n+1} - \lpar R_{n,p} \rpar^T \lambdab_{n+1}
} -  \innerprod{\pdv{L_{n+1,q}}{\theta_k}\wb_{n+1} + \pdv{R_{n,p}}{\theta_k}\wb_n}{\lambdab_{n+1}}. 
\end{multline} 
If we choose $\lambdab$ to satisfy the terminal condition and adjoint equations
\eqref{eq:terminal_condition} and \eqref{eq:adjoint_equations}, then we will
have
\begin{align*}
    \lpar \pdv{\jfunc_\NTime}{\wb_\NTime} \rpar^T + \lambdab_\NTime - (L_{\NTime,q})^T \lambdab_\NTime &= \zerob \\
    \lpar \pdv{\jfunc_n}{\wb_n}  \rpar^T + \lambdab_n - (L_{n,q})^T \lambdab_n
    - \lambdab_{n+1} - (R_{n,p})^T \lambdab_{n+1} &= \zerob
    ,\quad (n = 0,\dots,\NTime-1).
\end{align*}
Consequently all occurrences of $d\wb_n/d\theta_k$ in $d\Jcal/d\theta_k$ as
written in \eqref{eq:discrete_adjoint_before_cancellation} will have
their inner product taken with a quantity equal to $\zerob$, so that the only
remaining nonzero terms are exactly those given in \eqref{eq:gradient_accumulation}.
\end{proof}

\begin{algorithm}[htb!]
\caption{The HOHO Method / Discrete Adjoint}
\label{alg:discrete_adjoint}
\begin{algorithmic}
    \State Given the initial value problem $\dv{\wb}{t} = A(t;\thetab)\wb,\quad 0 \leq t \leq T, \quad \wb(0) = \wb_0,$ and the discrete objective functional $\Jcal = \sum_{n=1}^{\NTime} \jfunc(\wb_n)$, efficiently compute the gradient $\nabla \Jcal = [\inlinedv{\Jcal}{\theta_1}, \dots, \inlinedv{\Jcal}{\theta_{\revisionC[\NCtrl]{\NPar}}}]^T$ exactly with respect to the Hermite Runge-Kutta method using $q$ and $p$ derivatives for the  implicit/explicit parts.
    \State
    \State \textbf{Forward Evolution (Solve Forward Equations)}\Comment{See Algorithm \ref{alg:forward_timestep}.}
    \For{$n=1:\NTime$}
        \State Compute $\wb_n + R_{n,p}\wb_n$ explicitly.
        \State Solve $\wb_{n+1} - L_{n+1,q}\wb_{n+1} = \wb_n + R_{n,p}\wb_n$ for $\wb_{n+1}$ using the most efficient method available.
    \EndFor
    \State
    \State \textbf{Adjoint Evolution (Solve Terminal Condition and Adjoint Equations)}\Comment{See Algorithm \ref{alg:adjoint_evolution}.}
    \State Compute $-\left(\pdv{\jfunc_\NTime}{\wb_\NTime}\right)$ explicitly.
    \State Solve $\lambdab_\NTime - \lpar L_{\NTime,q}\rpar^T \lambdab_\NTime = -\left(\pdv{\jfunc_\NTime}{\wb_\NTime}\right)$ for $\lambdab_\NTime$ using the most efficient method available.
    \For{$n=\NTime-1:-1:1$}
    \State Compute $\lambdab_{n+1} + \lpar R_{n,p} \rpar^T\lambdab_{n+1}-\left(\pdv{\jfunc_n}{\wb_n}\right)$ explicitly.
    \State Solve $\lambdab_n - \lpar L_{n,q}\rpar^T \lambdab_n = \lambdab_{n+1} + \lpar R_{n,p} \rpar^T\lambdab_{n+1} -\left(\pdv{\jfunc_\NTime}{\wb_\NTime}\right)$ for $\lambdab_n$ using the most efficient method available.
    \EndFor
    \State
    \State \textbf{Gradient Accumulation}\Comment{See Algorithm \ref{alg:efficient_gradient} for more efficient gradient accumulation.}
    \State $\nabla \Jcal \gets \zerob$
    \State $(\nabla \Jcal)_k \gets (\nabla \Jcal)_k - \innerprod{\pdv{L_{n+1,q}}{\theta_k}\wb_{n+1} + \pdv{R_{n,p}}{\theta_k}\wb_n}{\lambdab_{n+1}}$ for $n=1,\dots,\NTime,\quad k=1,\dots,\NPar$
    \State \textbf{return } $\nabla \Jcal$
\end{algorithmic}
\end{algorithm}

\begin{remark}
    When the state is a matrix (e.g. a gate design problem), the gradient may be computed by performing the forward evolution, adjoint evolution, and gradient accumulation for each initial condition and summing the results. The computation is almost completely independent for each initial condition, except that computing $\inlinepdv{\jfunc_n}{\wb_n}$, where $\wb_n$ is the state vector after $n$ timesteps of a single initial condition may require the state vector after the $n$-th timestep for every initial condition. Typically, this is only the case for the final time objective $\jfunc_\NTime$ (e.g. the infidelity). Therefore, we can perfectly parallelize the forward evolution across the initial conditions, then solve the terminal condition equation \eqref{eq:terminal_condition}, then perform the rest of the adjoint evolution \eqref{eq:adjoint_equations} and gradient accumulation \eqref{eq:gradient_accumulation} in parallel.
\end{remark}

Unlike the forward and adjoint evolutions, performing the gradient accumulation \eqref{eq:gradient_accumulation} does not require solving any linear systems; it only requires explicit matrix-vector multiplications. But when the forward and adjoint evolutions are performed using a small timestep size $\Delta t$, a good initial guess for each timestep, and/or an effective preconditioner, the linear systems may be solved in only a few iterations, requiring only a few matrix-vector multiplications. In that case, because the number of matrix-vector multiplications required by the gradient accumulation as written in \eqref{eq:gradient_accumulation} scales with the number of control parameters and the number of timesteps, the gradient accumulation can be much more expensive than the forward and adjoint evolutions if handled naively. We now show how to perform the gradient accumulation \eqref{eq:gradient_accumulation} using a number of matrix-vector multiplications that is independent of the number of control parameters.
\begin{theorem} \label{thm:efficient_gradient_accumulation}
    When the time-dependence of $A(t)$ has the structure \eqref{eq:control_pulse_ansatz}, the gradient accumulation \eqref{eq:gradient_accumulation} can be rewritten so that the number of matrix-vector multiplications required is independent of the number of control parameters $\revisionC[\NCtrl]{\NPar}$.
\end{theorem}
\begin{proof}
    Expanding $L_{n+1,q}$ and $R_{n,p}$ in \eqref{eq:gradient_accumulation} (and using the fact that $D_{n,0}=I$ does not depend on $\thetab$), we get
    \begin{align}
        \dv{\Jcal}{\theta_k} &= 
        - \sum_{n=0}^{\NTime-1}\lpar 
        \innerprod{
            -\sum_{j=1}^q  (-1)^j c_j^{qp} \frac{1}{j!}  \pdv{D_{n+1,j}}{\theta_k}\wb_{n+1} + \sum_{j=1}^p c_i^{pq} \frac{1}{j!} \pdv{D_{n,j}}{\theta_k}\wb_n}
        {\lambdab_{n+1}} \rpar \\
        &=   \sum_{n=0}^{\NTime-1}\lpar 
         \sum_{j=1}^q (-1)^j c_j^{qp}\frac{1}{j!}\innerprod{\pdv{D_{n+1,j}}{\theta_k}\wb_{n+1}}{\lambdab_{n+1}} - \sum_{j=1}^p c_j^{pq}\frac{1}{j!}\innerprod{\pdv{D_{n,j}}{\theta_k}\wb_{n}}
        {\lambdab_{n+1}}
        \rpar. \label{eq:grad_accum_inner_products_separated}
    \end{align}
    To compute the inner products efficiently, we take $\partial / \partial \theta_k$ of \eqref{eq:linear_derivative_matrix}:
    \begin{equation*}
        \pdv{D_{n,0}}{\theta_k} = 0,\quad \pdv{D_{n,j+1}}{\theta_k} = \sum_{i=0}^j  \binom{j}{i} \pdv{\stepDeriv{A}{n}{j-i}}{\theta_k} D_{n,i} + \stepDeriv{A}{n}{j-i} \pdv{D_{n,i}}{\theta_k},\quad j=0,\dots,q-1,
    \end{equation*}
    and use the identity $\innerprod{M\ab}{\bb} = \innerprod{\ab}{M^T\bb}$ to rewrite the inner product as
    \begin{align} \label{eq:gradient_accumulation_recursion_1}
        \innerprod{\pdv{D_{n,j+1}}{\theta_k}\wb_n}{\lambdab_{n+1}} &= 
        \sum_{i=0}^j \binom{j}{i}
        \innerprod{ \pdv{\stepDeriv{A}{n}{j-i}}{\theta_k}D_{n,i}\wb_n}{\lambdab_{n+1}}
        + \binom{j}{i} \innerprod{ \stepDeriv{A}{n}{j-i}\pdv{D_{n,i}}{\theta_k} \wb_n}{\lambdab_{n+1}}
         \\
        &=
        \sum_{i=0}^j \binom{j}{i}
        \innerprod{ \pdv{\stepDeriv{A}{n}{j-i}}{\theta_k}D_{n,i}\wb_n}{\lambdab_{n+1}}
        + \binom{j}{i} \innerprod{ \pdv{D_{n,i}}{\theta_k} \wb_n}{\lpar \stepDeriv{A}{n}{j-i} \rpar^T \lambdab_{n+1}}
            . \label{eq:gradient_accumulation_recursion_2}
    \end{align}
    The inner products contatining $\partial \stepDeriv{A}{n}{j-i} / \partial \theta_k$ in \eqref{eq:gradient_accumulation_recursion_2} can be handled explicitly (and more efficiently by also storing each $D_{n,i}\wb_n$ to use when computing $D_{n,i+1}\wb_n$, as described in Section \ref{sec:computing_higher_order_derivatives}), while the inner products containing $\partial D_{n,i}/\partial \theta_k$ can be handled recursively by expanding the inner product using \eqref{eq:gradient_accumulation_recursion_1}. Ultimately, the only matrix-vector multiplications performed will be in the computation of the first argument of inner products of the form $\innerprod{\pdv{\stepDeriv{A}{n}{i}}{\theta_k}\xb}{\yb}$.

    The number of matrix-vector multiplications scales with the number of control parameters because the matrix $\inlinepdv{A_n^{(i)}}{\theta_k}$ may be different for each $\theta_k$. However, when $A(t)$ has the form \eqref{eq:control_pulse_ansatz} the inner product may be rewritten as
    \begin{equation} \label{eq:grad_accum_efficent_inner_product}
        \innerprod{\pdv{\stepDeriv{A}{n}{i}}{\theta_k}\xb}{\yb}
        = \sum_{j=1}^\NCtrl \pdv{c_j^{(i)}(t_n)}{\theta_k}\innerprod{A_{\ctrlA,j}\xb}{\yb}
        = \sum_{j=1}^\NCtrl \pdv{c_j^{(i)}(t_n)}{\theta_k}\innerprod{\xb}{A^T_{\ctrlA,j}\yb},
    \end{equation}
    which allows us to compute the inner product by performing only $\NCtrl$ (the number of controls) matrix-vector multiplications, independent of the number of control parameters, completing our proof. 
\end{proof} 

\begin{remark} \label{rmk:efficient_gradient_accumulation}
    Moving $A_{\ctrlA,j}$ to the second argument of the inner product in \eqref{eq:grad_accum_efficent_inner_product} allows us to compute the two inner products in \eqref{eq:grad_accum_inner_products_separated} using the same matrix-vector multiplications, since the inner products have the same second argument. I.e., the gradient accumulation requires us to compute inner products 
    \begin{equation*}
        \innerprod{\pdv{\stepDeriv{A}{n}{i}}{\theta_k}\xb_1}{\yb}
        = \sum_{j=1}^\NCtrl \pdv{c_j^{(i)}(t_n)}{\theta_k}\innerprod{\xb_1}{A^T_{\ctrlA,j}\yb}, \quad
        \innerprod{\pdv{\stepDeriv{A}{n+1}{i}}{\theta_k}\xb_2}{\yb} 
        = \sum_{j=1}^\NCtrl \pdv{c_j^{(i)}(t_{n+1})}{\theta_k}\innerprod{\xb_2}{A^T_{\ctrlA,j}\yb},
    \end{equation*}
    and by reusing $A_{c,j}^T \yb$ we can compute both inner products with only $\NCtrl$ matrix-vector multiplications. This halves the number of matrix-vector multiplications required for the gradient accumulation.
\end{remark}

The efficient gradient accumulation techniques detailed in Theorem \ref{thm:efficient_gradient_accumulation} and Remark \ref{rmk:efficient_gradient_accumulation} can by applied by using Algorithm \ref{alg:efficient_gradient}.

\begin{algorithm}[htb!]
    \caption{Efficient Gradient Accumulation}
    \label{alg:efficient_gradient}
\begin{algorithmic}
    \State Compute $\nabla_{\thetab}\innerprod{L_{n+1,q}\wb_{n+1} + R_{n,p}\wb_n}{\lambdab_{n+1}}$, in order to perform the gradient accumulation \eqref{eq:gradient_accumulation} efficiently when $A$ takes the usual form: $A(t;\thetab) = A_0 + c_1(t;\thetab)A_1 + \dots + c_\NCtrl(t;\thetab)A_\NCtrl$. Assumes $p=q$, so that the recursive procedures for computing $\nabla_{\thetab}\innerprod{L_{n+1,q}\wb_{n+1}}{\lambdab_{n+1}}$ and $\nabla_{\thetab}\innerprod{R_{n,q}\wb_{n}}{\lambdab_{n+1}}$ have the same depth.
    \State 
    \State Precompute $A_k\frac{\wb_{n+1}^{(i)}}{i!}$ and $A_k\frac{\wb_{n}^{(i)}}{i!}$ for $i=0,\dots,q$,\quad $k=1,\dots,\NCtrl$.
    \State \textbf{return } $\sum_{j=1}^q$ \Call{$\nabla$D}{$\lambdab_{n+1},\ j,\ (-1)^{j+1}c_j^{qp}\Delta t^j,\ c_j^{pq}\Delta t^j$}
    \State
    \Procedure{$\nabla$D}{$\lambdab, j, c_L, c_R$} $\longrightarrow c_L\nabla_{\thetab} \innerprod{L_{n+1,q}\wb_{n+1}}{\lambdab} + c_R \nabla_{\thetab} \innerprod{R_{n,p}\wb_n}{\lambdab_{n+1}} $
        \State \textbf{return}  $\frac{1}{j}\sum_{i=0}^{j-1}\left( \nabla \textrm{D}(\frac{1}{(j-1-i)!}A^{(j-1-i)}\lambdab,\ i,\ c_L,\ c_R) + \sum_{k=1}^\NCtrl  \frac{\nabla \ctrlA_{n+1,k}^{(j-1-i)}}{(j-1-i)!}\innerprod{A_k \frac{\wb_{n+1}^{(i)}}{i!}}{\lambdab} + \frac{\nabla \ctrlA_{n,k}^{(j-1-i)}}{(j-1-i)!}\innerprod{A_k \frac{\wb_{n}^{(i)}}{i!}}{\lambdab}\right)$
    \EndProcedure
\end{algorithmic}
\end{algorithm}

\begin{remark}
Applying $(L_{n,q})^T$ and $(R_{n,p})^T$ when solving the adjoint equations requires more matrix-vector multiplications than applying $L_{n,q}$ and $R_{n,p}$ in the forward evolution.
\end{remark}
Taking the transpose of \eqref{eq:linear_derivative_matrix}, we get
\begin{equation} \label{eq:linear_derivative_matrix_transpose}
    (D_{n,0})^T = I,\quad (D_{n,j+1})^T = \sum_{i=0}^j  \binom{j}{i} (D_{n,i})^T \lpar\stepDeriv{A}{n}{j-i}\rpar^T  ,\quad j=0,\dots,q-1.
\end{equation}
Whereas the computation of $D_{n,j+1} \xb$ can be done efficiently by storing and reusing $D_{n,i} \xb$ for $i=1,\dots,j$, in the computation of $(D_{n,j+1})^T \xb$ the vector being multiplied by each $(D_{n,i})^T$ changes with the values of $i$ and $j$. The vectors $\stepDeriv{A}{n}{i}\xb$, $i=0,\dots,j$ may be computed once and and reused. When this is done, the number of matrix-vector multiplications needed to multiply a vector by $(D_{n,j+1})^T$ is
\begin{equation*}
\cost[(D_{n,j+1})^T] = 1 + \sum_{i=0}^j \cost[(D_{n,j})^T],\quad j=1,\dots,q,
\end{equation*}
and the number of matrix-vector multiplications needed to multiply a vector by
$L_{n,q}$ or $R_{n,p}$ is
\begin{equation*}
    \cost[(L_{n,q})^T] = \sum_{i=0}^{q-1} \cost[(D_{n,j})^T],\quad 
    \cost[(R_{n,p})^T] = \sum_{i=0}^{p-1} \cost[(D_{n,j})^T].
\end{equation*}
Algorithm \ref{alg:adjoint_evolution} gives a recursive method for applying $(L_{n,q})^T$ and $(R_{n,p})^T$ when solving the adjoint equations.

A comparison of the number of matrix-vector multiplications required in the application of $L_{n,q}$ or $R_{n,p}$ (as in the forward evolution) and in the application of $(L_{n,q})^T$ or $(R_{n,p})^T$ (as in the adjoint evolution) when $q=p$ is given in Table \ref{tab:num_matrix_vec_multiplications}. The number of matrix-vector multiplications required in the adjoint evolution scales exponentially with the order of the method, limiting the efficiency of the extremely high-order methods. In the rest of this work, we consider method orders twelve and under, for which the number of matrix-vector multiplications required in the adjoint evolution does not exceed three times the number required in the forward evolution.

\begin{table}[htb]
\centering
\begin{tabular}{lcccccccc}
\toprule 
 & \multicolumn{8}{c}{Order of Method} \\
\cmidrule(lr){2-9} 
& 2 & 4 & 6 & 8 & 10 & 12 & 14 & 16\\
\hline
\# Matvec Mults: Forward & 1 & 3 & 6 & 10 & 15 & 21 & 28 & 36 \\
\# Matvec Mults: Adjoint & 1 & 3 & 7 & 15 & 31 & 63 & 127 & 255 \\
\bottomrule 
\end{tabular}
\caption{Number of matrix-vector multiplications required in the application of $L_{n,q}$ and $R_{n,p}$ (as in the forward evolution) and in the application of $(L_{n,q})^T$ and $(R_{n,p})^T$ (as in the adjoint evolution) when $q=p$.}
\label{tab:num_matrix_vec_multiplications}
\end{table}

\begin{algorithm}[htb!]
    \caption{Applying $-(L_{n,q})^T$ (or $(R_{n,p})^T$) to $\lambdab_n$.}
    \label{alg:adjoint_evolution}
\begin{algorithmic}
    \State For the system of ODEs $\dv{\wb}{t} = A(t;\thetab)\wb$, given $\lambdab_n$, the adjoint state corresponding to the $n$-th timestep, compute $-(L_{n,q})^T\lambdab_n$ (or $(R_{n,p})^T\lambdab_n$), as required when solving the teriminal condition \eqref{eq:terminal_condition} or adjoint equations \eqref{eq:adjoint_equations}.
    \State
    \State Precompute $\ \frac{1}{(j-1)!}(A_n^{(j-1)})^T\lambdab_n$ for $j=1,\dots,q$.
    \State \textbf{return } $\sum_{j=1}^q (-1)^j c_j^{qp} \Delta t^j \Call{DT}{\lambdab_n, j}$ \Comment{To apply $R_{n,p}$ instead of $L_{n,q}$, swap $p$ and $q$ and remove the $(-1)^j$ factor.}
    \State
    \Procedure{DT}{$\xb$, $j$} $\longrightarrow \frac{1}{j!}(D_{n,j})^T \xb$
        \If{$j = 0$}
            \State \textbf{return } $\xb$
        \Else \Comment{(When \textproc{DT} is not called recursively, then $\xb = \lambdab_n$ and $\frac{1}{(j-i)!}(A_n^{(j-i)})^T \xb$ is precomputed.)}
            \State \textbf{return } $\frac{1}{j}\sum_{i=0}^{j-1} \Call{DT}{\frac{1}{(j-i)!}(A_n^{(j-i)})^T \xb,i}$ 
        \EndIf
    \EndProcedure
\end{algorithmic}
\end{algorithm}

\begin{remark}
    Because the terminal condition \eqref{eq:terminal_condition} and adjoint equations \eqref{eq:adjoint_equations} require computing $(\inlinepdv{\jfunc_n}{\wb_n})^T$ for $n=\NTime,\dots,1$, it is necessary to either store the entire state history $\wb_0,\dots,\wb_\NTime$ when solving the forward equations, or to solve \Schrodinger's equation backward in time while solving the adjoint equations, so that we only every store a single state vector or adjoint state vector at each timestep. We refer to the latter approach as the ``memory-lean'' version. In this case, the gradient accumulation \eqref{eq:gradient_accumulation} is also performed alongside the adjoint evolution.
\end{remark}

There are two main drawbacks to the memory-lean version. 1) We must solve an additional system of ODEs, increasing the cost of our method by roughly 50\%. 2) Because the Hermite Runge-Kutta methods we use are non-reversible, the memory-lean version computed gradient will no longer be exact with respect to the discretized forward equations. \revision[The latter drawback can be especially detrimental to the performance of quasi-Newton optimizers.]{The latter drawback can be especially detrimental, depending on the magnitude of the error introduced by the memory-lean version, since the gradient becomes smaller as the objective function is minimized. As the length of the gradient decreases, eventually the approximation error dominates the computed value of the gradient. When using the GRAPE approach with an inexact gradient formula, de Fouquieres et al. \cite{de_Fouquieres_2011} observed that the convergence of the optimization slows down dramatically compared to the convergence when using an exact gradient as the optimization reaches high fidelities (the threshold at which the slowdown occurs increases with the accuracy of the gradient), effectively imposing a maximum fidelity achievable by the optimization algorithm.}

As a compromise, a checkpointing approach \cite{GriewankCheckpoint} could be used to store the state vector at only a \revision[subset of all timesteps and still computing an accurate gradient with respect to the discretized forward equations.]{fraction of all the timesteps but still compute an exact gradient. In order to recover the states between the checkpoints, the forward equations must be solved a second time while solving the adjoint equations. Similar to the memory-lean version, this increases the cost of the method by roughly 50\%. In the numerical examples that follow, we always store the entire state history.}

\subsection{Preconditioning}
We find that an efficient left preconditioner for the solution of the linear system \eqref{eq:timestep_matrices} is to apply $(I - \tilde{L}_{n+1,q})^{-1}$, where $\tilde{L}_{n+1,q}$ is the operator $L_{n+1,q}$ in the case where the time-dependent Hamiltonian $A(t)$ is replaced with the time-independent drift Hamiltonian $A_d$ (i.e. $A(t)$ when the control pulse amplitudes are all zero). Typically, the control pulse amplitudes are experimentally required to remain below a value which is significantly smaller than the eigenvalues of the drift Hamiltonian. Combined with the fact that $L_{n+1,q}$ scales with $\Delta t$, this suggests $L_{n+1,q}$ is only a small perturbation of $\tilde{L}_{n+1,q}$ for reasonably small $\Delta t$. When the drift Hamiltonian is diagonal, $(I - \tilde{L}_{n+1,q})^{-1}$ can be applied in $\mathcal{O}(N)$ time using Gaussian elimination, without ever forming $(I-\tilde{L}_{n+1,q})^{-1}$ explicitly. In this work, we consider only problems with diagonal drift Hamiltonians, and use the Gaussian elimination approach to preconditioning.

\section{Numerical Experiments}
\subsection{Correctness Experiments on a Rabi Oscillator Problem} \label{sec:correctness_experiments}
To verify the correctness of our numerical method, we  simulate a Rabi oscillator problem, which has an \revision[analytic]{analytical} solution. A Rabi oscillator problem consists of a two-level system in a rotating frame of reference. For this problem, \Schrodinger's equation is
\begin{equation}
\label{eq:rabi_osc_schrodinger}
    \dv{\psib}{t} = -i H_c \psib, \quad\text{where }
    H_c = \Omega a + \overline{\Omega} a^\dag = \re \Omega (a + a^\dagger) + i\im \Omega (a - a^\dagger), \quad \Omega \in \C.
\end{equation}
The analytic solution to \eqref{eq:rabi_osc_schrodinger} is $\psib(t) = U(t)\psib_0$, where the time-evolution unitary is given by
\revision[
\begin{equation} \label{eq:rabi_solution}
U(t) = 
\begin{bmatrix}
  \cos(|\Omega| t) &   (\sin(\theta) - i\cos(\theta)) \sin(|\Omega| t)\\
  -(\sin(\theta) + i\cos(\theta)) \sin(|\Omega| t) &   \cos(|\Omega| t)
\end{bmatrix},\quad\text{where } \Omega = |\Omega|(\cos(\theta)+i\sin(\theta)).
\end{equation}
]{
\begin{equation} \label{eq:rabi_solution}
U(t) = 
\begin{bmatrix}
  \cos(|\Omega| t) &   (\sin(\varphi) - i\cos(\varphi)) \sin(|\Omega| t)\\
  -(\sin(\varphi) + i\cos(\varphi)) \sin(|\Omega| t) &   \cos(|\Omega| t)
\end{bmatrix},\quad\text{where } \Omega = |\Omega|(\cos(\varphi)+i\sin(\varphi)).
\end{equation}
}
This solution can be derived using an eigenvector decomposition of the Hamiltonian \cite[Sec. 7.23]{petersson2021quantumDictionary}. The populations of the $\ket{0}$ and $\ket{1}$ states oscillate between 0 and 1 with period $\tau = \pi / |\Omega|$, \revisionB[and hence an application of control pulses $p_0 = \re(\Omega)$ and $q_0 = \im(\Omega)$ for duration $\tau$ is called a Rabi oscillation. A control pulse with $p_0=|\Omega|$, $q_0 = 0$, and duration $\tau/2$ implements a Pauli-$X$ gate.]{ a phenomenon known as Rabi oscillation.}

\subsubsection{Verification of Order of Accuracy}
Using the \revision[analytic]{analytical} solution \eqref{eq:rabi_solution}, we may evaluate the error in a numerical solution of the Rabi oscillator problem \eqref{eq:rabi_osc_schrodinger} at any point in time for any initial condition. In Table \ref{tab:rabi_osc}, we plot the Frobenius norm of the relative error in the final-time numerical solution of the Rabi oscillator problem for a basis of initial states $U_0 = \diag(1,1)$, with $\Omega = 0.05(\cos(\pi/4)+i\sin(\pi/4))$ (constant control pulses with amplitude $25/\sqrt{2}$ MHz), and $T = 9.5\pi / |\Omega|$ (nine and a half Rabi oscillations), starting with sixteen timesteps and then doubling the number of timesteps until we reach 256 timesteps. We observe that the method converges to the \revision[analytic]{analytical} solution with the expected orders of accuracy. Furthermore, we can simulate the problem with high accuracy using far fewer timesteps with the high-order methods than with the second-order method.

\begin{table}[htb!]
\centering
\begin{tabular}{rllllllllllll}
\toprule 
    & \multicolumn{2}{c}{Order 2} & \multicolumn{2}{c}{Order 4} & \multicolumn{2}{c}{Order 6} & \multicolumn{2}{c}{Order 8} & \multicolumn{2}{c}{Order 10} & \multicolumn{2}{c}{Order 12}\\
\cmidrule(lr){2-3} \cmidrule(lr){4-5} \cmidrule(lr){6-7} \cmidrule(lr){8-9} \cmidrule(lr){10-11} \cmidrule(lr){12-13}
Steps & Err & Cvg & Err & Cvg & Err & Cvg & Err & Cvg & Err & Cvg & Err & Cvg \\
\midrule 
16 & 4.5(-1) & - & 4.0(-1) & - & 1.1(-2) & - & 1.6(-4) & - & 1.4(-6) & - & 8.6(-9) & - \\
32 & 1.6 & 1.9 & 3.0(-2) & 3.7 & 1.9(-4) & 5.8 & 6.6(-7) & 7.9 & 1.4(-9) & 9.9 & 2.2(-12) & 11.9 \\
64 & 5.2(-1) & 1.7 & 1.9(-3) & 3.9 & 3.0(-6) & 6.0 & 2.6(-9) & 8.0 & 1.4(-12) & 10.0 & 3.4(-15) & 9.3 \\
128 & 1.3(-1) & 2.0 & 1.2(-4) & 4.0 & 4.7(-8) & 6.0 & 1.0(-11) & 8.0 & 1.5(-15) & 9.9 & 2.2(-15) & 0.7 \\
256 & 3.4(-2) & 2.0 & 7.7(-6) & 4.0 & 7.4(-10) & 6.0 & 3.5(-14) & 8.2 & 1.4(-15) & 0.1 & 4.6(-15) & 1.1 \\
\bottomrule 
\end{tabular}
    \caption{Frobenius norms and convergence rates of the final-time relative error in the numerical solution of nine and a half periods of Rabi oscillation, using the Hermite method with orders two through twelve.}
\label{tab:rabi_osc}
\end{table}

\subsubsection{Verification of Exact Gradients}
We can use the Rabi oscillator in an optimal control problem by taking $\thetab = [\revision[p_0, q_0]{\theta_1, \theta_2}]^T$ and $\Omega = \revision[p_0 + iq_0]{\theta_1 + i \theta_2}$, so that $H_c = \revision[p_0]{\theta_1}(a + a^\dag) + i\revision[q_0]{\theta_2}(a - a^\dag)$, and assigning some target gate $U^{\textrm{target}} \in \C^{2\times2}$. \revisionB{For example, taking $\theta_1 = |\Omega|,$ $\theta_2 = 0$, and using a gate duration of $\tau/2 = \pi/(2|\Omega|)$ implements a Pauli-$X$ gate.}

We can calculate the derivatives of each entry of $U(t)$ with respect to \revision[$p_0$ and $q_0$]{$\theta_1$ and $\theta_2$} by applying straightforward calculus to \eqref{eq:rabi_solution}:
\revision[
\begin{gather}
\label{eq:rabi_p_sensitivity}
\pdv{U(t)}{p_0} = \begin{bmatrix}
    -\frac{p_0 t}{|\Omega|}\sin(|\Omega|t)
    & \frac{-1}{p_0 - i q_0}\lpar \frac{q_0}{|\Omega|}\sin(|\Omega|t) + ip_0t\cos(|\Omega|t) \rpar 
    \\
\frac{1}{p_0+iq_0}\lpar \frac{q_0}{|\Omega|}\sin(|\Omega|t) - ip_0t\cos(|\Omega|t) \rpar
    &  -\frac{p_0 t}{|\Omega|}\sin(|\Omega|t)
\end{bmatrix},
\\
\label{eq:rabi_q_sensitivity}
\pdv{U(t)}{q_0} = \begin{bmatrix}
    -\frac{q_0 t}{|\Omega|}\sin(|\Omega|t)
    & \frac{1}{q_0 + i p_0}\lpar q_0t\cos(|\Omega|t) + \frac{i p_0}{|\Omega|}\sin(|\Omega|t) \rpar
    \\
    \frac{-1}{p_0+iq_0}\lpar \frac{p_0}{|\Omega|}\sin(|\Omega|t) + iq_0t\cos(|\Omega|t)\rpar
    &  -\frac{q_0 t}{|\Omega|}\sin(|\Omega|t)
\end{bmatrix}.
\end{gather}
]{
\begin{gather}
\label{eq:rabi_p_sensitivity}
\pdv{U(t)}{\theta_1} = \begin{bmatrix}
    -\frac{\theta_1 t}{|\Omega|}\sin(|\Omega|t)
    & \frac{-1}{\theta_1 - i \theta_2}\lpar \frac{\theta_2}{|\Omega|}\sin(|\Omega|t) + i\theta_1t\cos(|\Omega|t) \rpar 
    \\
\frac{1}{\theta_1+i\theta_2}\lpar \frac{\theta_2}{|\Omega|}\sin(|\Omega|t) - i\theta_1t\cos(|\Omega|t) \rpar
    &  -\frac{\theta_1 t}{|\Omega|}\sin(|\Omega|t)
\end{bmatrix},
\\
\label{eq:rabi_q_sensitivity}
\pdv{U(t)}{\theta_2} = \begin{bmatrix}
    -\frac{\theta_2 t}{|\Omega|}\sin(|\Omega|t)
    & \frac{1}{\theta_2 + i \theta_1}\lpar \theta_2t\cos(|\Omega|t) + \frac{i \theta_1}{|\Omega|}\sin(|\Omega|t) \rpar
    \\
    \frac{-1}{\theta_1+i\theta_2}\lpar \frac{\theta_1}{|\Omega|}\sin(|\Omega|t) + i\theta_2t\cos(|\Omega|t)\rpar
    &  -\frac{\theta_2 t}{|\Omega|}\sin(|\Omega|t)
\end{bmatrix}.
\end{gather}
}
For two matrices $A, B \in \C^{n \times n}$, $\tr(A^\dag B)$ can be rewritten as
\begin{equation*}
    \tr(A^\dag B) = \innerprod{A_\re}{B_\re}_F + \innerprod{A_\im}{B_\im}_F 
    + i \left[ \innerprod{A_\re}{B_\im}_F  - \innerprod{A_\im}{B_\re}_F\right].
\end{equation*}
The infidelity can be written as
\begin{equation*}
    \Infidelity = 1 - \frac{1}{4} \left[
        \left( \innerprod{U^{\textrm{target}}_\re}{U(T)_\re}_F + \innerprod{U^{\textrm{target}}_\im}{U(T)_\im}_F\right)^2 
    +  \left( \innerprod{U^{\textrm{target}}_\re}{U(T)_\im}_F  - \innerprod{U^{\textrm{target}}_\im}{U(T)_\re}_F\right)^2
    \right],
\end{equation*}
and then the derivatives of $\Infidelity$ may be computed as
\revision[
\begin{multline*}
    \dv{\Infidelity}{p_0/q_0} = 
    -\frac{1}{2} \bigg( \innerprod{U^{\textrm{target}}_\re}{U_{T,\re}}_F + \innerprod{U^{\textrm{target}}_\im}{U_{T,\im}}_F\bigg)
        \lpar \innerprod{U^{\textrm{target}}_\re}{\dv{U_{T,\re}}{p_0/q_0}}_F + \innerprod{U^{\textrm{target}}_\im}{\dv{U_{T,\im}}{p_0/q_0}}_F \rpar \\
    -\frac{1}{2}
    \bigg( \innerprod{U^{\textrm{target}}_\re}{U_{T,\im}}_F  - \innerprod{U^{\textrm{target}}_\im}{U_{T,\re}}_F\bigg)
    \left( \innerprod{U^{\textrm{target}}_\re}{\dv{U_{T,\im}}{p_0/q_0}}_F  - \innerprod{U^{\textrm{target}}_\im}{\dv{U_{T,\re}}{p_0/q_0}}_F\right),
\end{multline*}
]{
\begin{multline*}
    \dv{\Infidelity}{\theta_1/\theta_2} = 
    -\frac{1}{2} \bigg( \innerprod{U^{\textrm{target}}_\re}{U_{T,\re}}_F + \innerprod{U^{\textrm{target}}_\im}{U_{T,\im}}_F\bigg)
        \lpar \innerprod{U^{\textrm{target}}_\re}{\dv{U_{T,\re}}{\theta_1/\theta_2}}_F + \innerprod{U^{\textrm{target}}_\im}{\dv{U_{T,\im}}{\theta_1/\theta_2}}_F \rpar \\
    -\frac{1}{2}
    \bigg( \innerprod{U^{\textrm{target}}_\re}{U_{T,\im}}_F  - \innerprod{U^{\textrm{target}}_\im}{U_{T,\re}}_F\bigg)
    \left( \innerprod{U^{\textrm{target}}_\re}{\dv{U_{T,\im}}{\theta_1/\theta_2}}_F  - \innerprod{U^{\textrm{target}}_\im}{\dv{U_{T,\re}}{\theta_1/\theta_2}}_F\right),
\end{multline*}
}
where the \revision[sensitivites $\inlinedv{U_{T,\re/\im}}{p_0}$ and $\inlinedv{U_{T,\re/\im}}{q_0}$]{sensitivities $\inlinedv{U_{T,\re/\im}}{\theta_1}$ and $\inlinedv{U_{T,\re/\im}}{\theta_2}$} may be computed explicitly using \eqref{eq:rabi_p_sensitivity} and \eqref{eq:rabi_q_sensitivity}.

To test the correctness of our discrete adjoint implementation, we take the Hadamard gate as our target gate and again use \revision[$\Omega = p_0+iq_0 = 0.05(\cos(\pi/4)+i\sin(\pi/4))$ and $T = 9.5\pi / |\Omega|$]{
$\Omega = \theta_1+i\theta_2 = 0.05(\cos(\pi/4)+i\sin(\pi/4))$ and $T = 9.5\pi / |\Omega|$}. We compute the gradient using the discrete adjoint method for 16 to 256 timesteps and method order 2 through 12, and compare the Frobenius norm of the difference between the analytic gradient of the continuous problem and the gradient of the discrete problem using the discrete adjoint method. The results are shown in Table \ref{tab:rabi_gradient}. We observe that as the \revisionC[stepsize]{number of timesteps} increases, the gradient converges to the analytic gradient with machine precision.

\begin{table}[htb!]
\centering
\begin{tabular}{rllllllllllll}
\toprule 
    & \multicolumn{2}{c}{Order 2} & \multicolumn{2}{c}{Order 4} & \multicolumn{2}{c}{Order 6} & \multicolumn{2}{c}{Order 8} & \multicolumn{2}{c}{Order 10} & \multicolumn{2}{c}{Order 12}\\
\cmidrule(lr){2-3} \cmidrule(lr){4-5} \cmidrule(lr){6-7} \cmidrule(lr){8-9} \cmidrule(lr){10-11} \cmidrule(lr){12-13}
Steps & Err & Cvg & Err & Cvg & Err & Cvg & Err & Cvg & Err & Cvg & Err & Cvg \\
\midrule 
16 & 6.3 & - & 1.0(1) & - & 3.2(-1) & - & 4.6(-3) & - & 4.2(-5) & - & 2.6(-7) & - \\
32 & 7.9 & 0.3 & 8.8(-1) & 3.5 & 5.6(-3) & 5.8 & 2.0(-5) & 7.9 & 4.3(-8) & 9.9 & 6.6(-11) & 11.9 \\
64 & 1.2(1) & 0.6 & 5.8(-2) & 3.9 & 9.0(-5) & 6.0 & 7.8(-8) & 8.0 & 4.3(-11) & 10.0 & 8.3(-14) & 9.6 \\
128 & 3.9 & 1.7 & 3.6(-3) & 4.0 & 1.4(-6) & 6.0 & 3.1(-10) & 8.0 & 4.1(-15) & 13.3 & 4.9(-14) & 0.8 \\
256 & 1.0 & 2.0 & 2.3(-4) & 4.0 & 2.2(-8) & 6.0 & 1.1(-12) & 8.1 & 7.1(-14) & 4.1 & 1.3(-13) & 1.4 \\
\bottomrule 
\end{tabular}
    \caption{$\ell_2$ norms and convergence rates of the error in the discrete-adjoint-computed gradient of the infidelity for the Rabi oscillator problem with the Hadamard gate as the target. Numerical solution done over nine and a half periods of Rabi oscillation, using the Hermite method with orders two through twelve.}
\label{tab:rabi_gradient}
\end{table}

\subsection{Speedup Experiments on a Multi-Qudit Gate Design Problem} \label{sec:speedup_experiments}
As a practical example, we consider two qudits coupled to a resonator bus. Using the model \eqref{eq:dispersive_hamiltonian_model}, the drift Hamiltonian is 
\begin{equation} \label{eq:cnot3_hamiltonian}
\begin{multlined}
    H_d = \omega_1 a_1^\dag a_1 + \omega_2 a_2^\dag a_2  + \omega_R a_R^\dag a_R
        -\frac{\xi_1}{2} a_1^\dag a_1^\dag a_1 a_1 -\frac{\xi_2}{2}  a_2^\dag a_2^\dag a_2 a_2 -\frac{\xi_R}{2} a_R^\dag a_R^\dag a_R a_R \\
        - \xi_{21} a_2^\dag a_2 a_1^\dag a_1 - \xi_{R1}  a_R^\dag a_R a_1^\dag  a_1 - \xi_{R2} a_R^\dag a_R a_2^\dag a_2.
\end{multlined}
\end{equation}
In the model, the resonator is treated no differently than another qudit, but we use the index $R$ to indicate its significance as a resonator: the two qudits used for computation are weakly coupled to one another, but more strongly coupled to the resonator, which allows the resonator to mediate the interaction between the two qudits. For the physical parameters, we use the values from the CNOT3 example of Juqbox.jl \cite{petersson2022optimal}, which we provide in Table \ref{tab:cnot3_physical_parameters}.

\begin{table}[tb!]
\begin{subtable}[htb!]{\linewidth}
\centering
\begin{tabular}{cccccccccc} 
    \toprule
      & $\omega_1$ & $\omega_2$ & $\omega_R$ & $\xi_1$ & $\xi_2$ & $\xi_R$ & $\xi_{21}$ & $\xi_{R1}$  & $\xi_{R2}$  \\
     \midrule
    Value$ / 2\pi$ (GHz) & 4.11 & 4.82 & 7.84 & 2.20(-1) & 2.25(-1) & 2.83(-5) & 1.00(-6) & 2.49(-3) & 2.52(-3) \\
    \bottomrule
\end{tabular}
\caption{Physical parameters, rounded to three significant digits. The qubits are weakly coupled to one another, but more strongly coupled to the resonator.}
\label{tab:cnot3_physical_parameters}
\end{subtable}
\vspace{1em}

\begin{subtable}[htb!]{\linewidth}
\centering
\begin{tabular}{lccc}
  \toprule
    & Resonator & Qudit $1$ & Qudit $2$ \\
    \midrule
    \# Essential Levels & 1 & 2 & 2 \\
    \# Guard Levels     & 9 & 2 & 2 \\
  \bottomrule
\end{tabular}
\caption{Number of essential and guard levels for each subsystem.}
    \label{tab:subsystem_levels}
\end{subtable}
\caption{Parameters used to model the Hamiltonian of two qudits coupled to a resonator bus in \eqref{eq:cnot3_hamiltonian}.}
\end{table}

In the rotating frame (in which we will perform all numerical calculations), the drift Hamiltonian becomes 
\begin{equation*}
    H_d = -\frac{\xi_1}{2} a_1^\dag a_1^\dag a_1 a_1 -\frac{\xi_2}{2}  a_2^\dag a_2^\dag a_2 a_2 -\frac{\xi_R}{2} a_R^\dag a_R^\dag a_R a_R
        - \xi_{21} a_2^\dag a_2 a_1^\dag a_1 - \xi_{R1}  a_R^\dag a_R a_1^\dag  a_1 - \xi_{R2} a_R^\dag a_R a_2^\dag a_2,
\end{equation*}
and after making the rotating wave approximation, the control Hamiltonians are
\begin{equation*}
    H_{c,K}(t;\thetab) = p_K(t;\thetab)(a_K + a_K^\dag) + i q_K(t;\thetab)(a_K - a_K^\dag),\quad K \in \{1,2,R\}.
\end{equation*}

We model each of the two qubits as four-level subsystems \revisionC{(making them qu\emph{dits})}, and the resonator as a ten-level subsystem. We wish to implement a CNOT gate on the two qudits, utilizing the resonator to mediate the interaction. We use two essential levels and two guard levels for the qudits, and one essential level and nine guard levels for the resonator (see Table \ref{tab:subsystem_levels}). 
\revisionC{We also mimic Juqbox.jl's strategy of weighting the guard penalty for the guard states exponentially based on the distance from the highest level guard state. Let $N_G$ be the total number of guard levels in the computational basis (i.e., the dimension of the guard subspace), which is $N_G=156$ for this example. Let $N_{E,K}$ and $N_{G,K}$ be the number of essential and guard states in subsystem $K$ for $K \in \{1,2,R\}$. Then the weight of the guard penalty for the state $\ket{i_R i_2 i_1}$ is $w_G = \max\{w_{G,1}, w_{G,2}, w_{G,R}\}/N_G$, where $w_{G,K}$ are weights associated with the leakage into the guard subspace for each subsystem. They are defined by:
\[
w_{G,K} = \begin{cases}
0 & \textrm{if } i_K < N_{E,K},\\
0.001^{N_{E,K}+N_{G,K}-i_K-1} & \textrm{otherwise.}
\end{cases}
\]
For example, for the state $\ket{501}$ we get $w_{G,R} = 0.001^4$, $w_{G,1}=0$, and $w_{G,2}=0$, and therefore the weight of the guard penalty is $w_G = 0.001^4/156$. For the state $\ket{713}$ we get $w_{G,R} = 0.001^2,$ $w_{G,1} = 1$, and $w_{G,2}=0$, and therefore the weight of the guard penalty is $w_G = 1/156$. 

Weighting the guard levels in this way allows the optimization to populate the lower levels of the cavity, which may be advantageous for facilitating the interaction of the two qudits, while more heavily discouraging the population of the guard states which push the limits of the truncated Hilbert space used in the model.}

The gate is given in the laboratory frame in bra-ket notation by
\begin{equation*}
    \textrm{CNOT} = \ket{000}\bra{000} + \ket{001}\bra{001} + \ket{010}\bra{011} + \ket{011}\bra{010}.
\end{equation*}
Qudit 1 flips between the ground and excited states when qudit 2 is in the excited state. When qudit 2 is in the ground state, qudit 1 remains in the same state. The resonator always remains in the ground state (at the final time).

We fix the duration of the gate to 550 nanoseconds. As our control pulse ansatz, we use carrier waves multiplying degree 14 B-spline curves with 16 B-spline wavelets for the real and imaginary parts of the envelope, for a total of $3 \times 3 \times 2 \times 16 = 288$ control parameters. For each control Hamiltonian corresponding to a subsystem, we use three carrier wave frequencies: $0$, $-2\pi\xi_1$, and $-2\pi\xi_2$ for qudits $1$ and $2$, and $0$, $-2\pi\xi_{R1}$, and $-2\pi\xi_{R2}$ for the resonator. The B-spline degree of 14 ensures that the Hamiltonian is smooth enough that the methods up to order 12 will exhibit their full order of accuracy in the asymptotic region. For the lower-order methods, we could lower the degree of the B-spline curves to reduce the cost of evaluating the control functions while keeping the full order of accuracy, but we use the same degree for all orders for the sake of keeping the comparison as direct as possible.

\subsubsection{Determining Appropriate Step Size} \label{subsec:cnot3_stepsize}
In order for the optimization to find control pulses which implement a gate with low infidelity, the numerical solution of \Schrodinger's equation must be accurate enough to ensure that the error in the infidelity is reasonably smaller than the infidelity itself. With this in mind, before performing the optimization using the Hermite \revision[method of order $N$]{method of a given order,} we perform a numerical experiment to determine an appropriate stepsize $\Delta t$ to use in order to reach a target relative error in the numerical solution of \Schrodinger's at the final time of the gate.

For each order of the method, starting with two timesteps and doubling the number of timesteps each iteration until the state history cannot fit into main memory, we perform a forward evolution, record the final state, perform the rest of the gradient computation (adjoint evolution and gradient accumulation), and record the time taken for the full gradient computation. We do this for 25 samples of the control vector $\thetab$ from a uniform distribution with a range of $[-0.05,0.05]$ for each control parameter, which enforces a maximum amplitude of 50 MHz on the B-spline envelope for each carrier wave for each control pulse. For each order of the method, the same 25 samples are used, so that each method is attempting to simulate the exact same analytic dynamics. An absolute tolerance of $10^{-15}$ was used for the linear solves, which were performed using GMRES. This tolerance could be relaxed for greater computational efficiency, depending on the accuracy required of the state vector. In Figure \ref{fig:cnot3_stepsize}, we report the number of timesteps, mean time to compute the gradient, and mean relative error at the final time for each order of the method. The samples were collected using a single core of an Intel\textsuperscript{\textregistered} Xeon\textsuperscript{\textregistered} Gold 6148 CPU clocked at 2.40 GHz, and 32 GB of memory, in Michigan State University's high-performance computing cluster.

\begin{figure}[htb!]
    \centering
    \includegraphics[width=6.25in, height=4.5in]{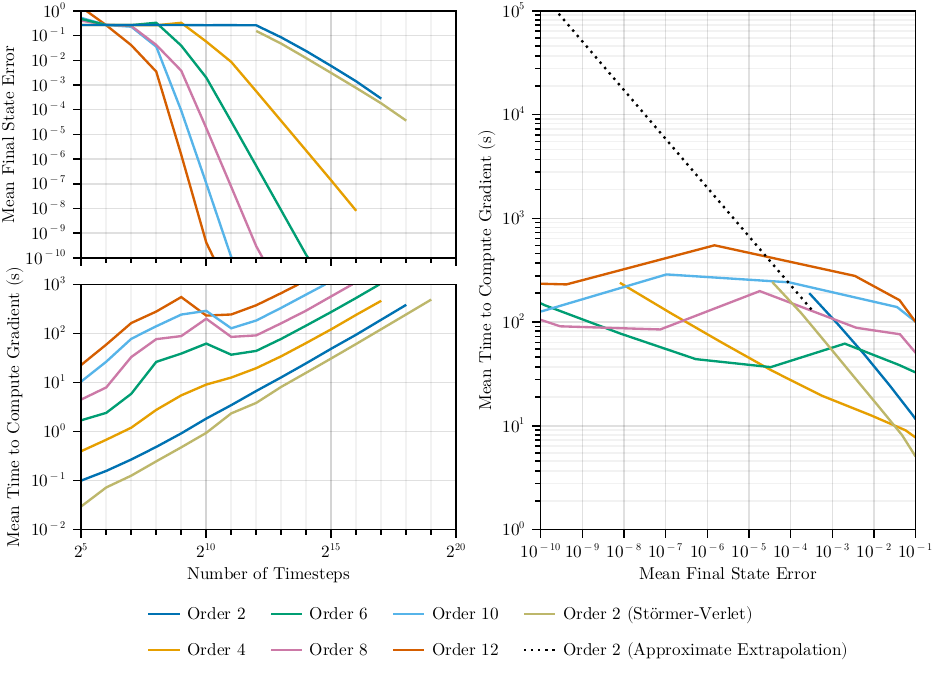} 
    \caption{Comparison of the number of timesteps used, the time taken to compute the gradient, and the relative error in the numerical approximation of the state, for the three subsystem example, starting in the initial states of the gate design problem: $\ket{000}, \ket{001}, \ket{010}$, and $\ket{011}$. The HOHO method is used with orders 2, 4, 6, 8, 10, and 12. We also compare with the second-order method St\"ormer-Verlet, as implemented in the Julia package Juqbox.jl.}
    \label{fig:cnot3_stepsize}
\end{figure}

In order to estimate the number of timesteps needed to reach a target relative error, we use linear interpolation of the number of timesteps and the mean relative error in the case that we have recorded mean relative errors above and below the target relative error. Otherwise, we estimate the number of timesteps using a least squares fit of the number of timesteps and the mean relative error in the region of asymptotic convergence.

This heuristic approach warrants some skepticism that we can use a single stepsize in order to consistently achieve a relative error below the target relative error during the optimization, since the control vector (and hence the dynamics) changes with each iteration of the optimization. However, in our experiment we observe that the standard deviation of the relative error is always an order of magnitude smaller than the mean relative error, indicating that the magnitude of the relative error is consistent for different values of $\thetab$. We use this numerical observation to justify using a single stepsize during the optimization. Table \ref{tab:cnot3_stepsize} shows the number of timesteps required for to reach several target relative final-time errors for each order of the Hermite method, and Table \ref{tab:cnot3_stepsize_speedup} shows the expected speedup of the gradient computation (to perform forward evolution, adjoint evolution, and gradient accumulation) for each order of the HOHO method compared to St\"ormer-Verlet, as implemented by Juqbox.jl, based on the average time taken to compute the gradient over the 25 different control vectors.

\begin{table}[htb!]
\begin{subtable}[htb!]{\linewidth}
\centering
\begin{tabular}{crrrrrrr}
\toprule 
& \multicolumn{7}{c}{Number of Timesteps Required} \\
\cmidrule(lr){2-8}
Target State Error & Order 2 (SV) & Order 2 & Order 4 & Order 6 & Order 8 & Order 10 & Order 12 \\
\midrule
1.0(-1) & 5,287 & 7,401 & 822 & 378 & 182 & 175 & 92 \\
1.0(-3) & 57,051 & 75,775 & 3,530 & 1,151 & 609 & 388 & 286 \\
1.0(-5) & 540,005 & 778,948 & 11,198 & 2,505 & 1,104 & 642 & 432 \\
1.0(-7) & 5,334,675 & 8,028,679 & 35,328 & 5,409 & 1,974 & 1,029 & 644 \\
\bottomrule 
\end{tabular}
\caption{Number of timesteps required to reach a target relative final-time error.}
\label{tab:cnot3_stepsize}
\end{subtable}
\vspace{1em}

\begin{subtable}[htb!]{\linewidth}
\centering
\begin{tabular}{crrrrrrr}
\toprule 
& \multicolumn{7}{c}{Speedup Factor} \\
\cmidrule(lr){2-8}
Target State Error & Order 2 (SV) & Order 2 & Order 4 & Order 6 & Order 8 & Order 10 & Order 12 \\
\midrule
1.0(-1) & 1.0 & 0.4 & 0.7 & 0.2 & 0.1 & 0.1 & 0.1 \\
1.0(-3) & 1.0 & 0.5 & 3.0 & 0.9 & 0.5 & 0.3 & 0.2 \\
1.0(-5) & 1.0 & 0.5 & 11.0 & 12.8 & 2.7 & 1.9 & 1.1 \\
1.0(-7) & 1.0 & 0.5 & 36.9 & 87.3 & 54.2 & 16.8 & 11.7 \\
\bottomrule 
\end{tabular}
\caption{Speedup factor (relative to St\"ormer-Verlet) for each order of the Hermite method.}
\label{tab:cnot3_stepsize_speedup}
\end{subtable}
\caption{\revision{The number of timesteps required and speedup factors observed when using St\"ormer-Verlet and each order of the Hermite method to numerically solve \Schrodinger's equation to a target relative final-time error for the three subsystem example, starting in the initial states of the gate design problem: $\ket{000}, \ket{001}, \ket{010}$, and $\ket{011}$. Results are based on the average accuracy and elapsed time for 25 randomly sampled control vectors.}}
\end{table}

In Tables \ref{tab:timesteps_excited} and \ref{tab:speedup_excited}, we again show the number of timesteps required and the average speedup for the three subsystem example, but for the evolution of a single initial state: $\ket{933}$, the most highly excited state of the system, for which the dynamics are expected to happen over much shorter timescales. This represents a worst-case scenario computationally, as more timesteps are required to resolve the faster dynamics, and consequently the higher-order methods are much more efficient, since their accuracy increases more significantly with the number of timesteps. Fast dynamics could also occur even in the computational subspace when there are many qudits, for example, or if the computation is done in the laboratory frame instead of the rotating frame.

More analytic approaches for determining the number of timesteps can be used, such as an analysis of the smallest frequency of the dynamics based on the eignevalues of the Hamiltonian \cite{petersson2022optimal}, but this numerical experiment is simple and captures the effect of changing the control vector on the numerical accuracy, resulting in a robust estimate of the number of timesteps needed to reach a particular level of accuracy. And although the experiment requires performing dozens of gradient computations, it is still cheap compared to the control optimization, which may perform hundreds or thousands of gradient computations.

\begin{table}[htb!]
\begin{subtable}[htb!]{\textwidth}
\centering
\begin{tabular}{crrrrrrr}
\toprule 
& \multicolumn{7}{c}{Number of Timesteps Required} \\
\cmidrule(lr){2-8}
Target State Error & Order 2 (SV) & Order 2 & Order 4 & Order 6 & Order 8 & Order 10 & Order 12 \\
\midrule
1.0(-1) & 232,401 & 311,607 & 14,684 & 4,495 & 2,292 & 1,424 & 1,041 \\
1.0(-3) & 2,004,693 & 2,267,052 & 46,606 & 9,755 & 4,136 & 2,351 & 1,556 \\
1.0(-5) & 17,452,973 & 16,493,625 & 146,904 & 21,050 & 7,381 & 3,754 & 2,316 \\
1.0(-7) & 151,946,628 & 119,997,072 & 463,615 & 45,367 & 13,143 & 5,967 & 3,413 \\
\bottomrule 
\end{tabular}
\caption{Number of timesteps required to reach a target relative final-time error.}
\label{tab:timesteps_excited}
\end{subtable}
\vspace{1em}

\begin{subtable}[htb!]{\linewidth}
\centering
\begin{tabular}{crrrrrrr}
\toprule 
& \multicolumn{7}{c}{Speedup Factor} \\
\cmidrule(lr){2-8}
Target State Error & Order 2 (SV) & Order 2 & Order 4 & Order 6 & Order 8 & Order 10 & Order 12 \\
\midrule
1.0(-1) & 1.0 & 0.7 & 5.0 & 5.5 & 1.2 & 0.9 & 0.8 \\
1.0(-3) & 1.0 & 0.9 & 15.4 & 26.9 & 25.1 & 9.3 & 6.3 \\
1.0(-5) & 1.0 & 1.2 & 45.1 & 121.4 & 142.9 & 101.4 & 52.5 \\
1.0(-7) & 1.0 & 1.5 & 136.7 & 530.1 & 774.6 & 702.4 & 459.1 \\
\bottomrule 
\end{tabular}
\caption{Speedup factor (relative to St\"ormer-Verlet) for each order of the Hermite method.}
\label{tab:speedup_excited}
\end{subtable}
\caption{\revision{The number of timesteps required and speedup factors observed when using St\"ormer-Verlet and each order of the Hermite method to numerically solve \Schrodinger's equation to a target relative final-time error for the three subsystem example, starting in the most highly excited state: $\ket{933}$. Results are based on the average accuracy and elapsed time for 25 randomly sampled control vectors.}}
\label{tab:excited}
\end{table}

\subsubsection{Performing Gate Optimization}
We now perform the gate optimization. We compare the performance of each order of the Hermite method for target final-time relative errors of $10^{-1}$, $10^{-3}$, $10^{-5}$, and $10^{-7}$. The optimization is managed by IPOPT (Interior Point OPTimizer) \cite{Wachter2006-li}, a software library for solving large scale nonlinear optimization problems using a primal-dual interior point method. We configure IPOPT to perform a limited-memory quasi-Newton approximation of the Hessian of the objective function for use in line searches (second-order optimization). IPOPT also allows for box constraints on the control vector, which we use to enforce our constraint on the control pulse amplitudes.

For each order and target error, we perform the optimization from ten initial control vectors, sampled from a uniform random distribution with a range of $[-0.005,0.005]$. The upper and lower bounds on each control parameter (enforced by IPOPT) are set to $\pm0.05$, enforcing a maximum amplitude of 50 MHz on the B-spline envelope of each carrier wave for each control pulse, the same as the upper and lower limits on the random distribution in the stepsize experiment. The same ten initial control vectors are used for all orders and all target errors, so that every optimization starts from the same analytic control pulses. Each optimization is allowed to run for six hours before being halted.

In Figure \ref{fig:cnot3_optimization}, for each target error we plot the generalized gate infidelity against the number of IPOPT iterations completed for each order of the Hermite method. In the figure, we do not plot the ``watchdog'' iterations, in which the optimizer aggressively explores new search directions, resulting in large changes in the objective function which would make the graph difficult to interpret (but the watchdog iterations are still counted in the total). These samples were collected using the same setup as the stepsize test, except we use four CPU cores instead of one and use multithreading to parallelize the forward and adjoint evolutions over the different initial and terminal conditions. In Figure \ref{fig:cnot3_bestcontrol}, we plot the time evolution of the population of the computational basis states for the control found that best minimized the generalized gate infidelity for a target final state error of $10^{-7}$\revision{, which were found using the sixth-order Hermite method with 5,409 timesteps, resulting in a stepsize of $\Delta t = 0.102$ nanoseconds (rounded to three significant digits)}.

\revision{Although the state populations in Figure \ref{fig:cnot3_bestcontrol} appear slowly varying, the dynamics of the \emph{probability amplitudes} (the complex-valued components of our state vector) exhibit highly oscillatory behavior. This can be seen in Figure \ref{fig:cnot3_bestcontrol_complex}, where we plot the time evolution of the probability amplitudes for one of the initial conditions of the gate design problem when using the ``best'' control. The fast oscillation of the probability amplitudes accounts for the computational difficulty of numerically solving \Schrodinger's equation for this problem.}

\begin{figure}[htb!]
    \centering
    \includegraphics[width=6.25in, height=3.25in]{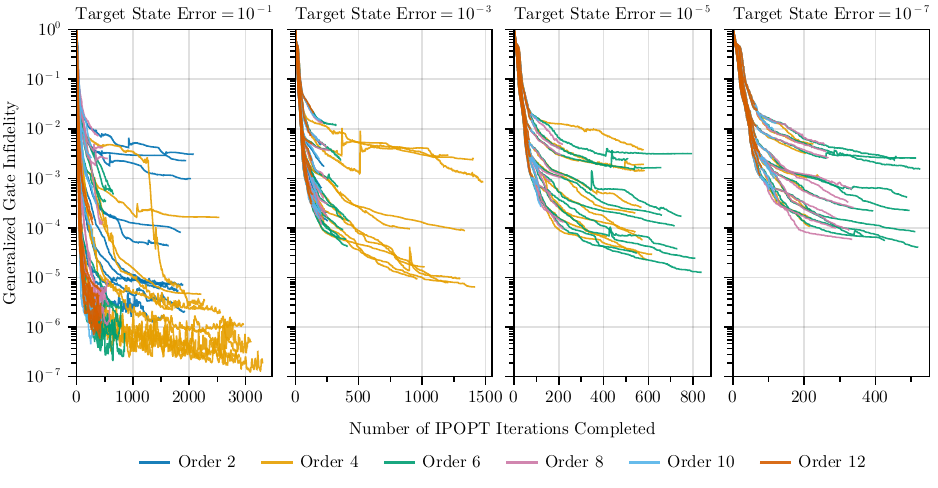}
    \caption{Optimization of a CNOT gate using two qudits coupled to a resonator bus as the mode, using B-spline curves multiplied carrier waves as the control pulse ansatz, using the methods and numbers of timesteps prescribed by Table \ref{tab:cnot3_stepsize} in order to reach target relative final state errors on the orders of $10^{-1}$, $10^{-3}$, $10^{-5}$, and $10^{-7}$. \revision{We do not include the second-order method for target errors of $10^{-5}$ and $10^{-7}$ due to the memory needed to store the state history for the large number of timesteps required. The second-order method could be used with checkpointing or the ``memory-lean'' approach, but it is clear that at these levels of accuracy the higher-order methods are significantly faster.}}
    \label{fig:cnot3_optimization}
\end{figure}

\begin{figure}[htb!]
    \centering
    \includegraphics[width=6.25in, height=5.5in]{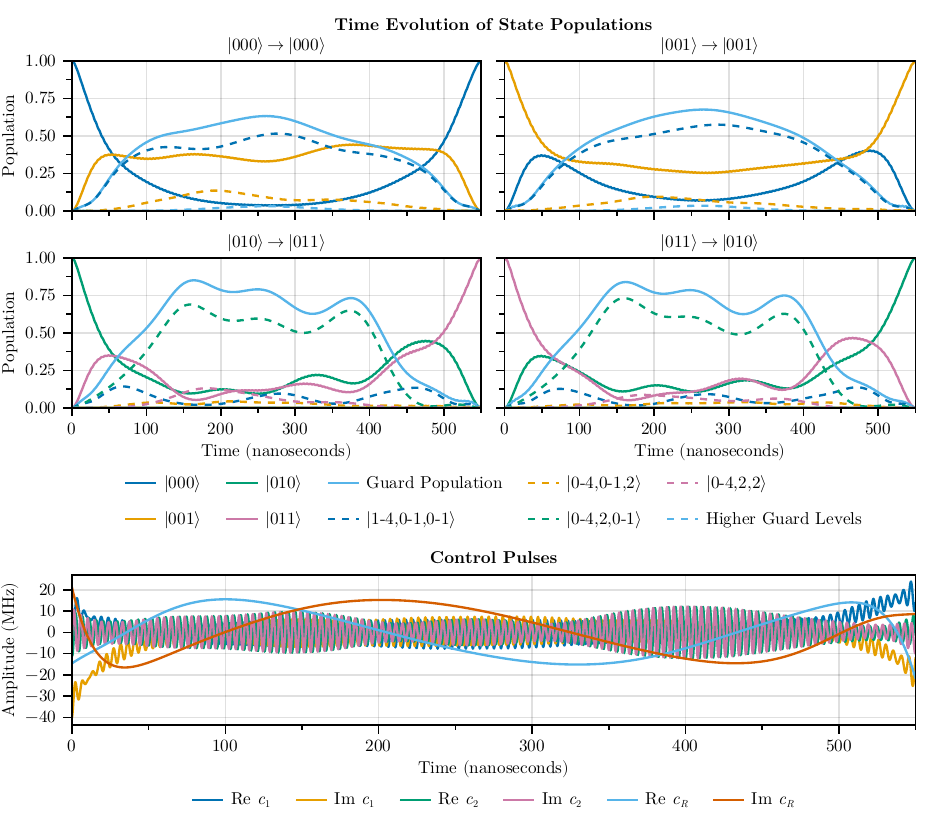}
    \caption{Time evolution of the state populations for the CNOT gate implemented by the best (lowest generalized infidelity) control pulses found by the optimizations performed with a target relative final state error of $10^{-7}$ for the two qudits plus resonator model (see Figure \ref{fig:cnot3_optimization}). \revision{The results shown here were obtained using the sixth-order Hermite method with 5,409 timesteps ($\Delta t \approx 0.102$ nanoseconds), which is the order and number of timesteps used in the optimization which found this control pulse.}}
    \label{fig:cnot3_bestcontrol}
\end{figure}

\begin{figure}[htb!]
    \centering
    \includegraphics[width=6.25in, height=4.5in]{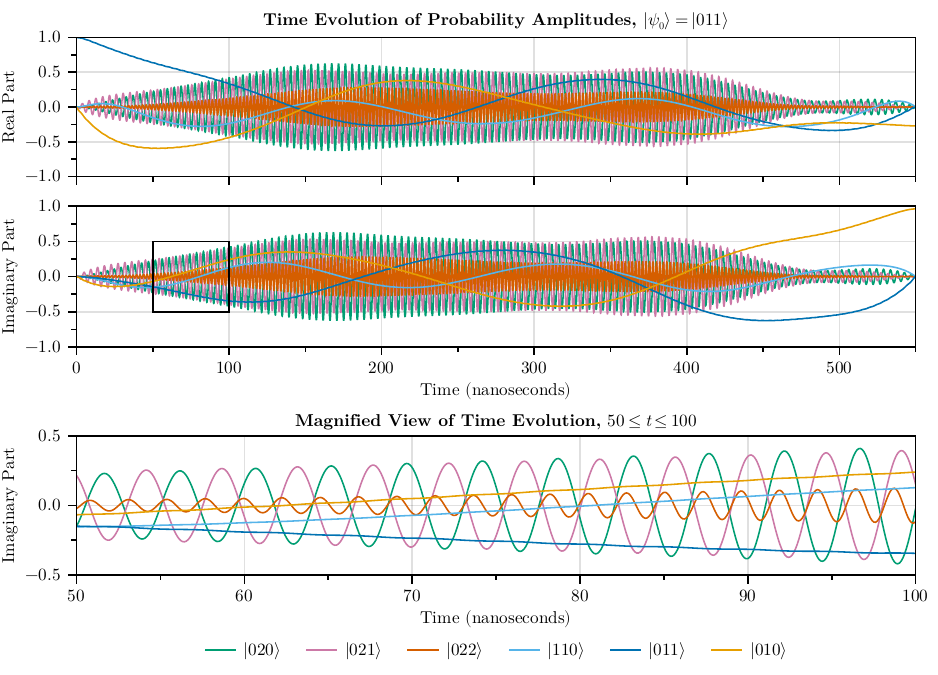}
    \caption{\revision{Time evolution of the complex-valued probability amplitudes of the six states with the highest population when starting from the $\ket{011}$ state and applying the ``best'' control pulse for implementing a CNOT gate in the two qudits plus resonator model (see Figure \ref{fig:cnot3_bestcontrol}). The results shown here were obtained using the sixth-order Hermite method with 5,409 timesteps ($\Delta t \approx 0.102$ nanoseconds), which is the order and number of timesteps used in the optimization which found this control pulse.}}
    \label{fig:cnot3_bestcontrol_complex}
\end{figure}

\subsubsection{Analysis}
From the plot of the mean relative error against the number of timesteps for the three subsystem example in Figure \ref{fig:cnot3_stepsize}, we see that each order of the Hermite method exhibits the full order of convergence. From the plot of the mean gradient computation time against the number of timesteps, we can see that the computation time increases with the number of timesteps, and the asymptotic behavior of each order of the method is the same, indicating that the computation time of the higher-order methods is roughly a constant multiple of the computation time of the lower-order methods, which is to be expected. Interestingly, for the sixth through twelfth-order methods the computation time \emph{decreases} around $2^9$--$2^{11}$ timesteps, for which the stepsize ranges from $1.07 \geq \Delta t \geq 0.27$ nanoseconds. We attribute this to a decrease in the number of GMRES iterations required in the linear solves of the forward and adjoint evolutions due to the effectiveness of our preconditioner and our initial guess (using Taylor expansion) when $\Delta t < 1$ nanosecond.

Most importantly, the plot of the mean gradient computation time against the mean relative error can be used to estimate which method order will be the fastest for a particular target error. For example, for a target error of $10^{-3}$, the gradient computation is fastest for the fourth-order method, but for a target error of $10^{-6}$ the sixth-order method is fastest. This also demonstrates the effectiveness of high-order methods; for a target error less than $10^{-2}$, the methods of order four and higher are always faster than the second-order methods. 

The computational efficiency of the high-order methods is showcased by Tables \ref{tab:cnot3_stepsize_speedup} and \ref{tab:speedup_excited}. When using the initial conditions from the gate design problem and trying to reach final state errors of $10^{-3}$--$10^{-5}$, using high-order methods results in average speedup factors of approximately $3$--$13$ compared to St\"ormer-Verlet (when choosing the most efficient order of the method for each target error). When the initial condition is the highly excited state $\ket{933}$, the average speedup factors are in the range $27$--$143$!

In addition to prescribing the number of timesteps required to reach a target error, Tables \ref{tab:cnot3_stepsize} and \ref{tab:timesteps_excited} also highlight another advantage of the high-order methods: reduced memory requirements. The memory required to store the state history in the forward and adjoint evolutions scales linearly in the number of timesteps and the dimension of Hilbert space. Because the high-order methods can obtain highly accurate solutions with a small number of timesteps, we are often able to store the entire state history in memory. This allows us to avoid using the ``memory-lean'' version of the gradient calculation, which would require more computation and produce less accurate gradients.

Figure \ref{fig:cnot3_optimization} highlights the effectiveness of high-order methods for the optimization process. We observe that for different order methods and different target errors, but starting from the same initial control vector, the generalized gate infidelity evolves roughly the same as a function of the number of IPOPT iterations. So our ability to optimize the objective function is determined mainly by the number of IPOPT iterations we are able to perform. Consequently, a method that can compute the gradient faster and therefore perform more IPOPT iterations in a fixed amount of wall time can converge to a lower objective function value. Because the higher-order methods can typically compute the gradient faster, they are able to optimize the objective function to a lower value in the same amount of time. As the target error decreases, the higher-order methods become more effective.

\section{Conclusions and Future Work} \label{sec:conclusions}
In this work, we derived a discrete adjoint method for efficiently computing gradients of objective functionals depending on the numerical solution of a linear system of ODEs solved by Hermite Runge-Kutta methods. Furthermore, we presented techniques for improving the computational efficiency of the method when the time-dependence of the Hamiltonian takes the typical form for systems controlled by control pulses, along with an efficient preconditioner for models where the drift Hamiltonian is diagonal.

We tested our method on the problem of designing a CNOT gate for two qudits coupled to a resonator bus, with guard levels included to ensure the validity of the model. We observed significant improvement in time and memory usage when using Hermite Runge-Kutta methods with the discrete adjoint method derived in this work (as implemented by QuantumGateDesign.jl) and comparing to the St\"ormer-Verlet method (as implemented by Juqbox.jl). When computing gradients to moderate levels of accuracy (relative final state error $10^{-3}$-$10^{-5}$) for the ``worst-case'' scenario that the system starts in its most highly excited state, we observed speedup factors of 27-143 and memory reduction factors of approximately $1,200$-$7,500$ by using high-order Hermite Runge-Kutta methods.

For a more strict relative final state error requirement of approximately $10^{-7}$, a speedup factor of 775 and a memory reduction factor of 11,561 was achieved by the most time-efficient method, and a speedup factor of 460 and a memory reduction factor of 44,520 was achieved by the most memory-efficient method. For stiffer systems, we expect high-order methods to be even more effective.

\revision{We also note the existence of \emph{quantum speed limits (QSL)}. The QSL defines the minimum evolution time it takes for a quantum system to dynamically evolve between two distinguishable quantum states \cite{Deffner_2017, Caneva_2009, Giovannetti_2003}. A related concept is the \emph{minimum control time (MCT)}, the minimum evolution time for which the optimization of the control pulses may (in principle) be successful, according to some metric, such as the infidelity reaching a certain value \cite{Poggi_2019}. In other words, given a parameterized Hamiltonian $H(t;\thetab)$ and a desired gate, the gate has a minimum duration. When the coupling between qubits is weak (e.g. small cross-Kerr values) and the control pulses are amplitude-bounded due to experimental constraints, the QSL and MCT for a given gate can be large. Another advantage of our method is that high-order methods can accurately simulate over long periods of time, whereas a low-order method would be more prone to drifting from the correct solution over time. Therefore, a particularly useful application of our method would be for solving quantum optimal control problems with a large time duration imposed by the quantum speed limit or minimum control time.}

\section*{Source Code Availability}
Source code for QuantumGateDesign.jl and scripts for running and plotting simulations in this paper are available as a GitHub
release for QuantumGateDesign.jl: \url{https://doi.org/10.5281/zenodo.17614033.}
 
\section*{Acknowledgments}
This material is based upon work supported by the National Science Foundation Graduate Research Fellowship Program under Grant No 2235783 (SL) and by NSF under Grant Number 2436319 (DA). Any opinions, findings, and conclusions or recommendations expressed in this material are those of the author(s) and do not necessarily reflect the views of the National Science Foundation. Part of this research was performed while one of the authors (SL) was visiting the Institute for Pure and Applied Mathematics (IPAM), which is supported by the National Science Foundation (Grant No. DMS-1925919). This material is based upon work supported by the U.S. Department of Energy, Office of Science, Advanced Scientific Computing Research (ASCR), under Award Number DE-SC0025424 (DA). This report was prepared as an account of work sponsored by an agency of the United States Government. Neither the United States Government nor any agency thereof, nor any of their employees, makes any warranty, express or implied, or assumes any legal liability or responsibility for the accuracy, completeness, or usefulness of any information, apparatus, product, or process disclosed, or represents that its use would not infringe privately owned rights. Reference herein to any specific commercial product, process, or service by trade name, trademark, manufacturer, or otherwise does not necessarily constitute or imply its endorsement, recommendation, or favoring by the United States Government or any agency thereof. The views and opinions of authors expressed herein do not necessarily state or reflect those of the United States Government or any agency thereof. We acknowledge Michigan State University's support through computational resources provided by the Institute for Cyber-Enabled Research. Thanks to the creators of the Julia programming language \cite{Julia_2017} and of the packages used in the creation of this paper, including Makie.jl \cite{Makie}, DataFrames.jl \cite{DataFrames}, and others which do not explicitly request citations in academic literature.

\bibliographystyle{elsarticle-num} 
\bibliography{bibliography}

@inproceedings{icosahom2014,
  title={Solving {PDEs} with {H}ermite interpolation},
  author={Hagstrom, Thomas and Appel{\"o}, Daniel},
  booktitle={Springer Lecture Notes in Computational Science and Engineering: Spectral and High Order Methods for Partial Differential Equations ICOSAHOM 2014: Selected papers from the ICOSAHOM conference, June 23-27, 2014, Salt Lake City, Utah, USA},
  pages={31--49},
  year={2015},
  organization={Springer}
}

@book{DQB,
	Author = {Dahlquist, Germund and Bj{\"o}rck, {\AA}ke},
	Isbn = {9780898716443 (v. 1 : hardcover : alk. paper)}, Publisher = {Society for Industrial and Applied Mathematics},
	Title = {Numerical methods in scientific computing},
	Year = 2008}

@article{petersson2022optimal,
author = {Anders Petersson, N. and Garcia, Fortino},
title = {Optimal Control of Closed Quantum Systems via {B}-Splines with Carrier Waves},
journal = {SIAM Journal on Scientific Computing},
volume = 44,
number = 6,
pages = {A3592-A3616},
year = 2022}

@article{CRAB_Caneva_2011,
	doi = {10.1103/physreva.84.022326},
	url = {https://doi.org/10.1103%2Fphysreva.84.022326},
	year = 2011,
	month = {aug},
	publisher = {American Physical Society ({APS})},
	volume = 84,
	number = 2,
	author = {Tommaso Caneva and Tommaso Calarco and Simone Montangero},
	title = {Chopped random-basis quantum optimization},
	journal = {Physical Review A}
}

@inproceedings{Corliss1998HOP_HermiteMethod,
  title={High-order stiff {ODE} solvers via automatic differentiation and rational prediction},
  author={Corliss, George F and Griewank, Andreas and Henneberger, Petra and Kirlinger, Gabriela and Potra, Florian A and Stetter, Hans J},
  booktitle={Numerical {A}nalysis and its {A}pplications: {F}irst {I}nternational {W}orkshop, {WNAA'96} Rousse, Bulgaria, June 24--26, 1996 Proceedings 1},
  pages={114--125},
  year=1997,
  organization={Springer}
}

@inproceedings{gu2020hermite,
  title={Hermite Methods in Time},
  author={Gu, Rujie and Hagstrom, Thomas},
  booktitle={Spectral and High Order Methods for Partial Differential Equations ICOSAHOM 2018: Selected Papers from the ICOSAHOM Conference, London, UK, July 9-13, 2018},
  pages={119--130},
  year=2020,
  organization={Springer}
}

@misc{betancourt2020discrete,
      title={The Discrete Adjoint Method: Efficient Derivatives for Functions of Discrete Sequences}, 
      author={Michael Betancourt and Charles C. Margossian and Vianey Leos-Barajas},
      year={2020},
      eprint={2002.00326},
      archivePrefix={arXiv},
      primaryClass={stat.CO}
}

@article{petersson2021quantumDictionary,
      title={Quantum Physics without the Physics}, 
      author={N. Anders Petersson and Fortino Garcia and Daniel Appel\"o and Stefanie Günther and Younsoo Choi and Ryan Vogt}, year={2021},
      journal = {arXiv:2012.03865},
      eprint={2012.03865},
      archivePrefix={arXiv},
      primaryClass={quant-ph}
}

@article{GRAFS,
  title = {Quantum optimal control via gradient ascent in function space and the time-bandwidth quantum speed limit},
  author = {Lucarelli, Dennis},
  journal = {Phys. Rev. A},
  volume = {97},
  issue = {6},
  pages = {062346},
  numpages = {7},
  year = {2018},
  month = {Jun},
  publisher = {American Physical Society},
  doi = {10.1103/PhysRevA.97.062346},
  url = {https://link.aps.org/doi/10.1103/PhysRevA.97.062346}
}

@article{SlepianSequences,
  author={Slepian, D.},
  journal={The Bell System Technical Journal}, 
  title={Prolate spheroidal wave functions, fourier analysis, and uncertainty — V: the discrete case}, 
  year={1978},
  volume={57},
  number={5},
  pages={1371-1430},
  doi={10.1002/j.1538-7305.1978.tb02104.x}
}

@article{GOAT_Machnes_2018,
   title={Tunable, Flexible, and Efficient Optimization of Control Pulses for Practical Qubits},
   volume={120},
   ISSN={1079-7114},
   url={http://dx.doi.org/10.1103/PhysRevLett.120.150401},
   DOI={10.1103/physrevlett.120.150401},
   number={15},
   journal={Physical Review Letters},
   publisher={American Physical Society (APS)},
   author={Machnes, Shai and Assémat, Elie and Tannor, David and Wilhelm, Frank K.},
   year={2018},
   month={apr}
}

@book{Griewank2008_AutomaticDifferentiationPrinciples,
author = {Griewank, Andreas and Walther, Andrea},
title = {Evaluating Derivatives},
publisher = {Society for Industrial and Applied Mathematics},
year = {2008},
doi = {10.1137/1.9780898717761},
edition = {Second},
URL = {https://epubs.siam.org/doi/abs/10.1137/1.9780898717761},
eprint = {https://epubs.siam.org/doi/pdf/10.1137/1.9780898717761}
}

@article{Margossian2019_AutomaticDifferentiationReview,
author = {Margossian, Charles C.},
title = {A review of automatic differentiation and its efficient implementation},
journal = {WIREs Data Mining and Knowledge Discovery},
volume = {9},
number = {4},
pages = {e1305},
doi = {https://doi.org/10.1002/widm.1305},
url = {https://wires.onlinelibrary.wiley.com/doi/abs/10.1002/widm.1305},
eprint = {https://wires.onlinelibrary.wiley.com/doi/pdf/10.1002/widm.1305},
year = {2019}
}

@article{Nelson2017_QOC_AD_GPU,
  title = {Speedup for quantum optimal control from automatic differentiation based on graphics processing units},
  author = {Leung, Nelson and Abdelhafez, Mohamed and Koch, Jens and Schuster, David},
  journal = {Phys. Rev. A},
  volume = {95},
  issue = {4},
  pages = {042318},
  numpages = {14},
  year = {2017},
  month = {Apr},
  publisher = {American Physical Society},
  doi = {10.1103/PhysRevA.95.042318},
  url = {https://link.aps.org/doi/10.1103/PhysRevA.95.042318}
}

@article{Abdelhafez2019_QOC_AD_QuantumTrajectories,
  title = {Gradient-based optimal control of open quantum systems using quantum trajectories and automatic differentiation},
  author = {Abdelhafez, Mohamed and Schuster, David I. and Koch, Jens},
  journal = {Phys. Rev. A},
  volume = {99},
  issue = {5},
  pages = {052327},
  numpages = {18},
  year = {2019},
  month = {May},
  publisher = {American Physical Society},
  doi = {10.1103/PhysRevA.99.052327},
  url = {https://link.aps.org/doi/10.1103/PhysRevA.99.052327}
}

@article{Schafer2020_QOC_AD,
doi = {10.1088/2632-2153/ab9802},
url = {https://dx.doi.org/10.1088/2632-2153/ab9802},
year = {2020},
month = {aug},
publisher = {IOP Publishing},
volume = {1},
number = {3},
pages = {035009},
author = {Frank Schäfer and Michal Kloc and Christoph Bruder and Niels Lörch},
title = {A differentiable programming method for quantum control},
journal = {Machine Learning: Science and Technology}
}

@article{Goerz2022_SemiAutomaticDifferentiation,
   title={Quantum Optimal Control via Semi-Automatic Differentiation},
   volume={6},
   ISSN={2521-327X},
   url={http://dx.doi.org/10.22331/q-2022-12-07-871},
   DOI={10.22331/q-2022-12-07-871},
   journal={Quantum},
   publisher={Verein zur Forderung des Open Access Publizierens in den Quantenwissenschaften},
   author={Goerz, Michael H. and Carrasco, Sebastián C. and Malinovsky, Vladimir S.},
   year={2022},
   month=dec, pages={871}
}

@article{GMRES,
author = {Saad, Youcef and Schultz, Martin H.},
title = {GMRES: A Generalized Minimal Residual Algorithm for Solving Nonsymmetric Linear Systems},
journal = {SIAM Journal on Scientific and Statistical Computing},
volume = {7},
number = {3},
pages = {856-869},
year = {1986},
doi = {10.1137/0907058},
URL = {https://doi.org/10.1137/0907058},
eprint = {https://doi.org/10.1137/0907058},
}

@misc{gunther2021quandaryopensourcecpackage,
      title={Quandary: An open-source {C}++ package for high-performance optimal control of open quantum systems}, 
      author={Stefanie Günther and N. Anders Petersson and Jonathan L. Dubois},
      year={2021},
      eprint={2110.10310},
      archivePrefix={arXiv},
      primaryClass={quant-ph},
      url={https://arxiv.org/abs/2110.10310}, 
}

@inproceedings{trowbridge2023directcollocationquantumoptimal,
  author={Trowbridge, Aaron and Bhardwaj, Aditya and He, Kevin and Schuster, David I. and Manchester, Zachary},
  booktitle={2023 IEEE International Conference on Quantum Computing and Engineering (QCE)}, 
  title={Direct Collocation for Quantum Optimal Control}, 
  year={2023},
  volume={01},
  number={},
  pages={1278-1285},
  doi={10.1109/QCE57702.2023.00144}
}

@book{deboorSplines2009,
  address = {New York},
  author = {d. Boor, Carl},
  publisher = {Springer Verlag},
  title = {A Practical Guide to Splines},
  year = 1978
}

@misc{petersson2024timeparallelmultipleshootingmethodlargescale,
      title={A time-parallel multiple-shooting method for large-scale quantum optimal control}, 
      author={N. Anders Petersson and Stefanie Günther and Seung Whan Chung},
      year={2024},
      eprint={2407.13950},
      archivePrefix={arXiv},
      primaryClass={quant-ph},
      url={https://arxiv.org/abs/2407.13950}, 
}

@misc{raftery2017directdigitalsynthesismicrowave,
      title={Direct digital synthesis of microwave waveforms for quantum computing},
      author={J. Raftery and A. Vrajitoarea and G. Zhang and Z. Leng and S. J. Srinivasan and A. A. Houck},
      year={2017},
      eprint={1703.00942},
      archivePrefix={arXiv},
      primaryClass={quant-ph},
      url={https://arxiv.org/abs/1703.00942},
}

@article{Sch_fer_2018,
   title={Fast quantum logic gates with trapped-ion qubits},
   volume={555},
   ISSN={1476-4687},
   url={http://dx.doi.org/10.1038/nature25737},
   DOI={10.1038/nature25737},
   number={7694},
   journal={Nature},
   publisher={Springer Science and Business Media LLC},
   author={Schäfer, V. M. and Ballance, C. J. and Thirumalai, K. and Stephenson, L. J. and Ballance, T. G. and Steane, A. M. and Lucas, D. M.},
   year={2018},
   month=mar, pages={75–78} 
}

@article{Levine_2018,
   title={High-Fidelity Control and Entanglement of {R}ydberg-Atom Qubits},
   volume={121},
   ISSN={1079-7114},
   url={http://dx.doi.org/10.1103/PhysRevLett.121.123603},
   DOI={10.1103/physrevlett.121.123603},
   number={12},
   journal={Physical Review Letters},
   publisher={American Physical Society (APS)},
   author={Levine, Harry and Keesling, Alexander and Omran, Ahmed and Bernien, Hannes and Schwartz, Sylvain and Zibrov, Alexander S. and Endres, Manuel and Greiner, Markus and Vuletić, Vladan and Lukin, Mikhail D.},
   year={2018},
   month=sep
}

@article{Theis_2016,
   title={High-fidelity {R}ydberg-blockade entangling gate using shaped, analytic pulses},
   volume={94},
   ISSN={2469-9934},
   url={http://dx.doi.org/10.1103/PhysRevA.94.032306},
   DOI={10.1103/physreva.94.032306},
   number={3},
   journal={Physical Review A},
   publisher={American Physical Society (APS)},
   author={Theis, L. S. and Motzoi, F. and Wilhelm, F. K. and Saffman, M.},
   year={2016},
   month=sep
}

@article{Krantz_2019,
   title={A quantum engineer’s guide to superconducting qubits},
   volume={6},
   ISSN={1931-9401},
   url={http://dx.doi.org/10.1063/1.5089550},
   DOI={10.1063/1.5089550},
   number={2},
   journal={Applied Physics Reviews},
   publisher={AIP Publishing},
   author={Krantz, P. and Kjaergaard, M. and Yan, F. and Orlando, T. P. and Gustavsson, S. and Oliver, W. D.},
   year={2019},
   month=jun
}

@article{KHANEJA2005296,
title = {Optimal control of coupled spin dynamics: design of NMR pulse sequences by gradient ascent algorithms},
journal = {Journal of Magnetic Resonance},
volume = {172},
number = {2},
pages = {296-305},
year = {2005},
issn = {1090-7807},
doi = {https://doi.org/10.1016/j.jmr.2004.11.004},
url = {https://www.sciencedirect.com/science/article/pii/S1090780704003696},
author = {Navin Khaneja and Timo Reiss and Cindie Kehlet and Thomas Schulte-Herbrüggen and Steffen J. Glaser}
}

@article{de_Fouquieres_2011,
   title={Second order gradient ascent pulse engineering},
   volume={212},
   ISSN={1090-7807},
   url={http://dx.doi.org/10.1016/j.jmr.2011.07.023},
   DOI={10.1016/j.jmr.2011.07.023},
   number={2},
   journal={Journal of Magnetic Resonance},
   publisher={Elsevier BV},
   author={de Fouquieres, P. and Schirmer, S.G. and Glaser, S.J. and Kuprov, Ilya},
   year={2011},
   month=oct, pages={412–417}
}

@article{Eitan_2011,
  title = {Optimal control with accelerated convergence: Combining the {K}rotov and quasi-{N}ewton methods},
  author = {Eitan, Reuven and Mundt, Michael and Tannor, David J.},
  journal = {Phys. Rev. A},
  volume = {83},
  issue = {5},
  pages = {053426},
  numpages = {10},
  year = {2011},
  month = {May},
  publisher = {American Physical Society},
  doi = {10.1103/PhysRevA.83.053426},
  url = {https://link.aps.org/doi/10.1103/PhysRevA.83.053426}
}

@article{Wachter2006-li,
  title     = "On the implementation of an interior-point filter line-search
               algorithm for large-scale nonlinear programming",
  author    = "W{\"a}chter, Andreas and Biegler, Lorenz T",
  journal   = "Math. Program.",
  publisher = "Springer Science and Business Media LLC",
  volume    =  106,
  number    =  1,
  pages     = "25--57",
  month     =  mar,
  year      =  2006,
  language  = "en"
}

@article{QuantumStateTomography,
  title = {Measurement of qubits},
  author = {James, Daniel F. V. and Kwiat, Paul G. and Munro, William J. and White, Andrew G.},
  journal = {Phys. Rev. A},
  volume = {64},
  issue = {5},
  pages = {052312},
  numpages = {15},
  year = {2001},
  month = {Oct},
  publisher = {American Physical Society},
  doi = {10.1103/PhysRevA.64.052312},
  url = {https://link.aps.org/doi/10.1103/PhysRevA.64.052312}
}

@article{Magesan_2012,
   title={Efficient Measurement of Quantum Gate Error by Interleaved Randomized Benchmarking},
   volume={109},
   ISSN={1079-7114},
   url={http://dx.doi.org/10.1103/PhysRevLett.109.080505},
   DOI={10.1103/physrevlett.109.080505},
   number={8},
   journal={Physical Review Letters},
   publisher={American Physical Society (APS)},
   author={Magesan, Easwar and Gambetta, Jay M. and Johnson, B. R. and Ryan, Colm A. and Chow, Jerry M. and Merkel, Seth T. and da Silva, Marcus P. and Keefe, George A. and Rothwell, Mary B. and Ohki, Thomas A. and Ketchen, Mark B. and Steffen, M.},
   year={2012},
   month=aug }

@article{Givoli_2021_adjoint,
     title={A tutorial on the adjoint method for inverse problems}, 
     volume={380},
     ISSN={0045-7825},
     DOI={https://doi.org/10.1016/j.cma.2021.113810},
     abstractNote={This paper is a basic tutorial on the adjoint method when used in a computational scheme for solving an inverse problem. The adjoint method is a technique for the efficient calculation of the gradient of the functional which is to be minimized in the solution process. The method is presented in a slightly non-standard way, which is believed to be simpler and less abstract than the common presentation found in most books and papers, yet equally general. More specifically, the unknown parameters are discretized from the outset while all other variables remain continuous. The adjoint method is applied here both at the continuous level and at the discrete level. Both steady-state (elliptic) and time-dependent problems are considered. Various computational aspects are discussed.},
     journal={Computer Methods in Applied Mechanics and Engineering},
     author={Givoli, Dan},
     year={2021},
     pages={113810}
 }

@article{Julia_2017,
    title={Julia: A fresh approach to numerical computing},
    author={Bezanson, Jeff and Edelman, Alan and Karpinski, Stefan and Shah, Viral B},
    journal={SIAM {R}eview},
    volume={59},
    number={1},
    pages={65--98},
    year={2017},
    publisher={SIAM},
    doi={10.1137/141000671},
    url={https://epubs.siam.org/doi/10.1137/141000671}
}

@article{Makie,
  doi = {10.21105/joss.03349},
  url = {https://doi.org/10.21105/joss.03349},
  year = {2021},
  publisher = {The Open Journal},
  volume = {6},
  number = {65},
  pages = {3349},
  author = {Simon Danisch and Julius Krumbiegel},
  title = {{Makie.jl}: Flexible high-performance data visualization for {Julia}},
  journal = {Journal of Open Source Software}
}

@article{DataFrames,
 title={DataFrames.jl: Flexible and Fast Tabular Data in {J}ulia},
 volume={107},
 url={https://www.jstatsoft.org/index.php/jss/article/view/v107i04},
 doi={10.18637/jss.v107.i04},
 number={4},
 journal={Journal of Statistical Software},
 author={Bouchet-Valat, Milan and Kamiński, Bogumił},
 year={2023},
 pages={1--32}
}

@article{GriewankCheckpoint,
author = {Griewank, Andreas and Walther, Andrea},
title = {Algorithm 799: {R}evolve: an implementation of checkpointing for the reverse or adjoint mode of computational differentiation},
year = {2000},
issue_date = {March 2000},
publisher = {Association for Computing Machinery},
address = {New York, NY, USA},
volume = {26},
number = {1},
issn = {0098-3500},
url = {https://doi.org/10.1145/347837.347846},
doi = {10.1145/347837.347846},
journal = {ACM Trans. Math. Softw.},
month = mar,
pages = {19–45},
numpages = {27}
}

@article{nocedal2006numerical,
  title={Numerical optimization},
  author={Nocedal, J and Wright, SJ},
  journal={Springer Series in Operations Research and Financial Engineering},
  year={2006}
}

@Book{Nielsen-Chuang,
  author = {M. Nielsen and I. Chuang},
  title = {Quantum computation and quantum information},
  publisher = {Cambridge University Press},
  year = {2000}
}

@misc{Lee_QuantumGateDesign_jl_2025,
    author = {Lee, Spencer and Appel\"o, Daniel and Garcia, Fortino},
    license = {MIT},
    title = {{Q}uantum{G}ate{D}esign.jl},
    howpublished = {Github},
    url = {https://github.com/leespen1/QuantumGateDesign.jl},
    version = {0.2.0},
    year = {2025},
    doi = {https://doi.org/10.5281/zenodo.17614033}
}

@book{DerivativeFreeOptimization,
author = {Conn, Andrew R. and Scheinberg, Katya and Vicente, Luis N.},
title = {Introduction to Derivative-Free Optimization},
publisher = {Society for Industrial and Applied Mathematics},
year = {2009},
doi = {10.1137/1.9780898718768},
address = {},
edition   = {},
URL = {https://epubs.siam.org/doi/abs/10.1137/1.9780898718768},
eprint = {https://epubs.siam.org/doi/pdf/10.1137/1.9780898718768}
}

@misc{margossian2022efficientautomaticdifferentiationimplicit,
      title={Efficient Automatic Differentiation of Implicit Functions}, 
      author={Charles C. Margossian and Michael Betancourt},
      year={2022},
      eprint={2112.14217},
      archivePrefix={arXiv},
      primaryClass={stat.CO},
      url={https://arxiv.org/abs/2112.14217}, 
}

@misc{hovland2024differentiatinglinearsolvers,
      title={Differentiating Through Linear Solvers}, 
      author={Paul Hovland and Jan Hückelheim},
      year={2024},
      eprint={2404.17039},
      archivePrefix={arXiv},
      primaryClass={cs.MS},
      url={https://arxiv.org/abs/2404.17039}, 
}

@article{Giovannetti_2003,
   title={Quantum limits to dynamical evolution},
   volume={67},
   ISSN={1094-1622},
   url={http://dx.doi.org/10.1103/PhysRevA.67.052109},
   DOI={10.1103/physreva.67.052109},
   number={5},
   journal={Physical Review A},
   publisher={American Physical Society (APS)},
   author={Giovannetti, Vittorio and Lloyd, Seth and Maccone, Lorenzo},
   year={2003},
   month=may
}

@article{Caneva_2009,
   title={Optimal Control at the Quantum Speed Limit},
   volume={103},
   ISSN={1079-7114},
   url={http://dx.doi.org/10.1103/PhysRevLett.103.240501},
   DOI={10.1103/physrevlett.103.240501},
   number={24},
   journal={Physical Review Letters},
   publisher={American Physical Society (APS)},
   author={Caneva, T. and Murphy, M. and Calarco, T. and Fazio, R. and Montangero, S. and Giovannetti, V. and Santoro, G. E.},
   year={2009},
   month=dec }

@article{Deffner_2017,
   title={Quantum speed limits: from Heisenberg’s uncertainty principle to optimal quantum control},
   volume={50},
   ISSN={1751-8121},
   url={http://dx.doi.org/10.1088/1751-8121/aa86c6},
   DOI={10.1088/1751-8121/aa86c6},
   number={45},
   journal={Journal of Physics A: Mathematical and Theoretical},
   publisher={IOP Publishing},
   author={Deffner, Sebastian and Campbell, Steve},
   year={2017},
   month=oct, pages={453001} }

@article{Poggi_2019,
   title={Geometric quantum speed limits and short-time accessibility to unitary operations},
   volume={99},
   ISSN={2469-9934},
   url={http://dx.doi.org/10.1103/PhysRevA.99.042116},
   DOI={10.1103/physreva.99.042116},
   number={4},
   journal={Physical Review A},
   publisher={American Physical Society (APS)},
   author={Poggi, Pablo M.},
   year={2019},
   month=apr }

@book {HairerGeometric,
    AUTHOR = {Hairer, Ernst and Lubich, Christian and Wanner, Gerhard},
     TITLE = {Geometric numerical integration},
    SERIES = {Springer Series in Computational Mathematics},
    VOLUME = {31},
   EDITION = {Second},
      NOTE = {Structure-preserving algorithms for ordinary differential
              equations},
 PUBLISHER = {Springer-Verlag, Berlin},
      YEAR = {2006},
     PAGES = {xviii+644},
      ISBN = {3-540-30663-3; 978-3-540-30663-4},
}

@techreport{LLNLCompositionalReport,
  author       = {Lee, Spencer and Guenther, Stefanie and Petersson, N. Anders},
  title        = {Compositional Methods for Schrödinger's Equation with Application to Optimal Control},
  institution  = {Lawrence Livermore National Laboratory (LLNL), Livermore, CA (United States)},
  doi          = {10.2172/1888108},
  url          = {https://www.osti.gov/biblio/1888108},
  place        = {United States},
  year         = {2022},
  month        = {08}}

\end{document}